\documentclass[11pt]{article}
\usepackage{layout}
\usepackage[makeroom]{cancel}
\usepackage{bbm}
\usepackage[affil-it]{authblk}

\usepackage{mathtools}
\usepackage{comment}

\mathtoolsset{showonlyrefs}
\usepackage[nottoc,notlot,notlof]{tocbibind}
\usepackage{amsfonts,dsfont}
\usepackage{amsmath}
\usepackage{amssymb,bm}
\usepackage{amsthm}
\usepackage{graphicx} \usepackage{enumerate} \usepackage{multicol}
\usepackage{mathrsfs} \usepackage[all,cmtip]{xy}

\usepackage[normalem]{ulem} 

\usepackage[shortlabels]{enumitem}
\setlist[enumerate,1]{wide, labelindent=0pt,label={\upshape(\roman*)}}
\usepackage{comment}

\def\Xint#1{\mathchoice
{\XXint\displaystyle\textstyle{#1}}%
{\XXint\textstyle\scriptstyle{#1}}%
{\XXint\scriptstyle\scriptscriptstyle{#1}}%
{\XXint\scriptscriptstyle\scriptscriptstyle{#1}}%
\!\int}
\def\XXint#1#2#3{{\setbox0=\hbox{$#1{#2#3}{\int}$ }
\vcenter{\hbox{$#2#3$ }}\kern-.6\wd0}}

\def\dashint{\Xint-}

\usepackage[rgb,dvipsnames]{xcolor}
\usepackage{float}
\usepackage{ulem}
\usepackage{circuitikz}

\usepackage[colorlinks,linkcolor=blue,citecolor=blue,pagebackref,hypertexnames=false, breaklinks]{hyperref}

\usepackage{float}

\allowdisplaybreaks[4]

\newlength{\bibitemsep}
\newlength{\bibparskip}
\let\oldthebibliography\thebibliography
\renewcommand\thebibliography[1]{%
 \oldthebibliography{#1}%
 \setlength{\parskip}{\bibitemsep}%
 \setlength{\itemsep}{\bibparskip}%
}

\newtheorem{theorem}{Theorem}[section]

\newtheorem{definition}[theorem]{Definition}
\newtheorem{proposition}[theorem]{Proposition}

\newtheorem{corollary}[theorem]{Corollary}
\newtheorem{lemma}[theorem]{Lemma}
\newtheorem{remark}[theorem]{Remark}
\newtheorem{example}[theorem]{Example}
\newtheorem{examples}[theorem]{Examples}
\newtheorem{foo}[theorem]{Remarks}

\newtheorem{open question}[theorem]{Open Question}
\newtheorem{c/p}[theorem]{Conjecture/Proposition}
\newenvironment{Example}{\begin{example}\rm}{\end{example}}

\newcommand{\brak}[1]{\left(#1\right)} 

\def\Xint#1{\mathchoice
    {\XXint\displaystyle\textstyle{#1}}%
    {\XXint\textstyle\scriptstyle{#1}}%
    {\XXint\scriptstyle\scriptscriptstyle{#1}}%
    {\XXint\scriptscriptstyle\scriptscriptstyle{#1}}%
    \!\int}
    \def\XXint#1#2#3{{\setbox0=\hbox{$#1{#2#3}{\int}$}
    \vcenter{\hbox{$#2#3$}}\kern-.5\wd0}}

\def\vint{\mathop{\mathchoice%
 {\setbox0\hbox{$\displaystyle\intop$}\kern 0.22\wd0%
 \vcenter{\hrule width 0.6\wd0}\kern -0.82\wd0}%
 {\setbox0\hbox{$\textstyle\intop$}\kern 0.2\wd0%
 \vcenter{\hrule width 0.6\wd0}\kern -0.8\wd0}%
 {\setbox0\hbox{$\scriptstyle\intop$}\kern 0.2\wd0%
 \vcenter{\hrule width 0.6\wd0}\kern -0.8\wd0}%
 {\setbox0\hbox{$\scriptscriptstyle\intop$}\kern 0.2\wd0%
 \vcenter{\hrule width 0.6\wd0}\kern -0.8\wd0}}%
 \mathopen{}\int}

\newcommand{\pip}{\varphi}

\newcommand{\CS}{\mathcal S}

\newcommand{\eps}{\varepsilon}

\usepackage[textwidth=3.2cm]{todonotes}

\title{Korevaar-Schoen and heat kernel characterizations \\ of Sobolev and BV spaces on local trees}
\author{Fabrice Baudoin\footnote{F.B.'s research is partially funded by Villum Fonden through a Villum Investigator grant and by grant 10.46540/4283-00175B from Independent Research Fund Denmark}, Li Chen\footnote{L.C.'s research is partially funded by grant 10.46540/4283-00175B from Independent Research Fund Denmark}, Meng Yang}
\date{}

\begin{document}

\maketitle

\begin{abstract}
 We study Sobolev and BV spaces on local trees which are metric spaces locally isometric to real trees. Such spaces are equipped with a Radon measure satisfying a locally uniform volume growth condition. Using the intrinsic geodesic structure, we define weak gradients and develop from it a coherent theory of Sobolev and BV spaces. We provide two main characterizations: one via Korevaar–Schoen type energy functionals and another via the heat kernel associated with the natural Dirichlet form. Applications include interpolation results for Besov–Lipschitz spaces, critical exponent computations, and a Nash inequality. In globally tree-like settings, we also establish $L^p$ gradient bounds for the heat semigroup.
\end{abstract}

\tableofcontents

\section{Introduction}

The study of analysis on metric measure spaces has become a central theme in modern analysis, driven by the need to extend classical calculus and partial differential equations (PDEs)  beyond the smooth setting of Euclidean spaces or Riemannian manifolds. Fundamental objects in this pursuit are Besov or Sobolev spaces and spaces of functions of bounded variation, which provide the natural framework for variational problems, geometric measure theory, and the analysis of PDEs in non-smooth environments. Defining and characterizing these spaces on general metric measure spaces $(X, d, m)$ presents significant challenges due to the lack of readily available differential structures. Various approaches have been developed, including Haj{\l}asz-Sobolev spaces \cite{MR1401074}, Newtonian spaces based on upper gradients \cite{MR3363168}, and definitions via heat semigroup regularisation \cite{BV2,BV3,BV1}, Korevaar-Schoen type energy functionals \cite{Bau22,kajino2024korevaarschoenpenergyformsassociated} or $p$-energy forms \cite{kajino2025penergyformsfractalsrecent,murugan2025firstordersobolevspacesselfsimilar,sasaya2025constructionpenergymeasuresassociated}.

In this paper, we focus on a specific, yet rich, class of metric measure spaces known as local trees. A metric space $(X, d)$ is a local tree if it is locally isometric to a real tree (Definition 2.2). Real trees themselves are characterized by the existence of unique geodesics between any two points and the absence of loops (Definition 2.1). The class of local trees encompasses a diverse range of structures, including classical trees, metric graphs or cable systems (Example 2.4), and certain fractal-like spaces such as the Sierpi{\'n}ski carpet cable system (Figure 1) or modifications of the Vicsek set (Figure 4, Example 2.5). These spaces, while possessing a clear local 1-dimensional structure along their "branches" or "edges", can exhibit complex global topology and singularities, such as  Hausdorff dimensions  greater than 1.

The analysis on such spaces requires some basic assumptions. We consider a local tree $(X, d)$ equipped with a Radon measure $m$ that satisfies a uniform volume control condition: $m(B(x,r)) \approx \Phi(r)$ for some doubling function $\Phi$ and small $r$. Crucial to our analysis, the notion of gradient is tied to the intrinsic 1-dimensional structure. Functions are defined to be absolutely continuous (AC) if their variation along local geodesics can be represented by integration against a weak gradient $\partial f$ (Definition 2.7). This weak gradient naturally lives on the "skeleton" $\mathcal S$ of the space (the union of all non-trivial geodesic open segments) and is measured with respect to the natural length measure $\nu$ (Section 2.1). The Sobolev space $W^{1,p}(X, m)$ then consists of functions $f \in L^p(X, m)$ whose weak gradient $\partial f$ belongs to $L^p(\mathcal S, \nu)$ (Definition 2.13). The space $BV(X, m)$ is defined via $L^1$-relaxation of the 1-energy $\int_\mathcal{S} |\partial f| \, d\nu$ (Definition 2.17). Notably, in this general framework the reference measure $m$ and the length measure $\nu$ can be mutually singular (e.g., $m(\mathcal S)=0$ is possible).

The central aim of this work is to provide robust  characterizations of these intrinsically defined Sobolev $W^{1,p}$ and BV spaces on local trees satisfying  uniform local volume control and a uniform local tree property (meaning the local isometry to a tree holds uniformly up to a certain scale). Such characterizations are valuable as they connect the intrinsic definition to other analytical viewpoints, offer alternative (sometimes more practical) ways to  estimate norms, and facilitate the transfer of techniques from different areas of analysis. We establish two main types of characterizations: one based on Korevaar-Schoen (KS) type energy functionals and another based on heat kernel estimates.

The first main contribution is the Korevaar-Schoen type characterization. Originating from the work of Korevaar and Schoen \cite{MR1266480} in the context of harmonic maps, this approach characterizes smoothness or regularity by measuring the averaged $p$-th power of function differences at small scales. We define specific KS energy functionals

$$E_{p,\Psi_p}(f, r)=\frac{1}{\Psi_p(r)}\int_X \frac{1}{m(B(x,r))}\int_{B(x,r)} |f(y)-f(x)|^p dm(y) dm(x)$$
where $\Psi_p(r) = r^{p-1}\Phi(r)$ is related to the volume growth (Section 3.3). We prove that for $p > 1$, the $W^{1,p}(X, m)$ norm is equivalent to the Korevaar-Schoen norm (Theorem 3.7). A corresponding result holds for $p = 1$, characterizing the BV seminorm $||\partial f||(X)$ in terms of $E_{1,\Psi_1}(f, r)$ (Theorem 3.8). Establishing these equivalences requires significant technical work; in particular, we construct controlled partitions of unity adapted to the local tree structure and the volume measure (Lemma 3.3, Theorem 3.5), which are of independent interest. Along the way we establish the weak monotonicity property of the KS functional:
\[
\sup_{r \in (0,r_0]} E_{p,\Psi_p}(f, r) \le C \liminf_{r \to 0} E_{p,\Psi_p}(f, r).
\]
It is remarkable that such weak monotonicity property holds assuming only the uniform volume control at small scale since this usually requires $p$-capacity  estimates type assumptions, see \cite{Bau22, murugan2025firstordersobolevspacesselfsimilar,shimizu2025characterizationssobolevfunctionsbesovtype}. This means that in our setting of local trees those capacity type estimates automatically follow from the uniform volume assumption, a situation which is similar to the smooth Riemannian setting. The key lemma illustrating this phenomenon is Lemma 3.3 which actually only relies on the space being a uniformly locally doubling local tree. 

The second main contribution lies in the heat kernel characterization. This approach connects the function spaces to the behavior of the heat semigroup $e^{t\Delta}$ associated with the natural energy form $\mathcal E_2(f) = \int_\mathcal{S} |\partial f|^2 \, d\nu$. Following \cite{KumagaiPRIMS}, we first establish detailed estimates for the corresponding heat kernel $p_t(x, y)$. We derive two-sided sub-Gaussian-type bounds reflecting the metric $d$ and the volume growth $\Phi$, specifically relating the time scale $t$ to a length scale via $r = \Psi_2^{-1}(t)$ where $\Psi_2(r) = r\Phi(r)$ (Theorem 4.3, 4.5, 4.7, Section 4.1). These estimates include upper bounds, near-diagonal lower bounds, and escape rate estimates. With these kernel estimates in hands, we introduce the heat kernel based norm and  prove that the $W^{1,p}$ (for $p > 1$) and BV (for $p = 1$) norms are equivalent to the heat kernel based norms (Theorem 4.8). This result places the analysis on local trees within the broader framework of heat kernel based function space theory developed in recent years \cite{BV1,BV2,BV3}.

Furthermore, we explore applications of these characterizations. We use the KS characterization to determine critical exponents for Besov-Lipschitz spaces $B^{\alpha,p}(X, m)$ on these local trees and establish real interpolation results, showing that these Besov spaces form an interpolation scale between $L^p$ and $W^{1,p}$ (or BV for $p=1$) (Theorem 3.11). We also derive a Nash inequality adapted to the volume growth (Theorem 3.12). Finally, we discuss how these results can be refined in the specific case where $(X, d)$ is a global tree with global volume estimates (Section 5). In that global tree case, we also establish $L^p$ gradient estimates for the heat semigroup.

We structure the paper as follows. Section 2 introduces local trees and defines the relevant function spaces (AC, $W^{1,p}$, BV) and energy forms. Section 3 develops the Korevaar-Schoen type characterization, including the construction of partitions of unity and applications to interpolation and Nash inequalities. Section 4 focuses on the heat kernel approach, establishing the necessary kernel estimates and proving the heat kernel based characterization of $W^{1,p}$ and BV spaces. Section 5 discusses extensions and refinements specific to the global tree setting.

\section{Sobolev spaces on local trees}

\subsection{Local trees}

We first recall the following definition of (real) tree.

\begin{definition}
 A metric space $(X,d)$ is said to be a tree  if it satisfies the following two properties:
\begin{enumerate}
    \item \textbf{Unique geodesic.} For all $u, v \in X$ there exists a unique isometric embedding $\phi_{u,v} : [0, d(u, v)] \to X$  such that $\phi_{u,v} (0) = u$ and $\phi_{u,v} (d(u, v)) = v$.
\item  \textbf{Loop-free}. For every injective continuous map $\kappa : [0, 1] \to  X$ one has: $$\kappa([0, 1]) = \phi_{\kappa(0),\kappa(1)}([0, d(\kappa(0), \kappa(1))]).$$
\end{enumerate}
\end{definition}

There exists a vast literature about trees, see for instance \cite{AEW,Chiswell} and references therein.  In this paper we are interested in spaces locally isometric to trees.

\begin{definition}
 A metric space $(X,d)$ is called a local tree if it is locally isometric to a real tree, that is for every $x \in X$, there exists $\iota (x) >0$ such that $X \cap B(x,\iota (x))$  is a real tree. 
\end{definition} 

\begin{example}
Any  tree is of course a local tree.
\end{example}

\begin{example}[Cable systems]\label{cable system example}

Intuitively, a cable system  is a connected  collection of intervals which are glued together at the vertices of a graph. More precisely,  a cable system $\mathbf{G}$ is composed of a set of vertices $V$, a set of (internal) edges $\mathbf{E}$ and a set of rays $\mathbf{R}$. The graph $(V,\mathbf{E})$ is assumed to be connected, locally finite and simple. For each edge $e\in \mathbf{E} $, there are two endpoints $e^-$ and $e^+$ in $V$ as well as a length $r(e)>0$. Each ray has one associated endpoint $e^-$ in $V$ and the length is infinite.   For $e\in\mathbf{E}$ let $I_e = [0,r(e)]$ and if $e\in\mathbf{R}$ then $I_e = [0,\infty)$. In this case the metric graph $\mathbf{G}$  is the metric space  obtained by replacing the edge $e$ by an isometric copy of $I_e$ and glue them in an obvious way at the vertices. The metric  $d$ on  $\mathbf{G}$ is the obvious geodesic path metric on $\mathbf{G}$.  It is then easy to see that $(\mathbf{G},d)$ is a local tree. We refer to  \cite[Section 2.4]{MR4009428} for further details on cable systems. The figure \ref{SC cable example} illustrates an example of cable system: the Sierpi\'nski carpet cable system, see \cite[Section 9]{gao2024holderregularityharmonicfunctions} for details on this example.

\begin{figure}
\centering
\begin{tikzpicture}[scale=0.3]

\foreach\x in {0,6,12}{
\foreach\y in {0,12}{

\foreach\xx in {0,2,4}{
\foreach\yy in {0,4}{
\draw(\x+\xx+0,\y+\yy+0)--(\x+\xx+2,\y+\yy+0)--(\x+\xx+2,\y+\yy+2)--(\x+\xx+0,\y+\yy+2)--cycle;

\draw[fill=black] (\x+\xx+0,\y+\yy+0) circle (0.25);
\draw[fill=black] (\x+\xx+0,\y+\yy+2) circle (0.25);
\draw[fill=black] (\x+\xx+2,\y+\yy+2) circle (0.25);
\draw[fill=black] (\x+\xx+2,\y+\yy+0) circle (0.25);

\draw[fill=black] (\x+\xx+1,\y+\yy+0) circle (0.25);
\draw[fill=black] (\x+\xx+2,\y+\yy+1) circle (0.25);
\draw[fill=black] (\x+\xx+1,\y+\yy+2) circle (0.25);
\draw[fill=black] (\x+\xx+0,\y+\yy+1) circle (0.25);
}
}

\foreach\xx in {0,4}{
\foreach\yy in {2}{
\draw(\x+\xx+0,\y+\yy+0)--(\x+\xx+2,\y+\yy+0)--(\x+\xx+2,\y+\yy+2)--(\x+\xx+0,\y+\yy+2)--cycle;

\draw[fill=black] (\x+\xx+0,\y+\yy+0) circle (0.25);
\draw[fill=black] (\x+\xx+0,\y+\yy+2) circle (0.25);
\draw[fill=black] (\x+\xx+2,\y+\yy+2) circle (0.25);
\draw[fill=black] (\x+\xx+2,\y+\yy+0) circle (0.25);

\draw[fill=black] (\x+\xx+1,\y+\yy+0) circle (0.25);
\draw[fill=black] (\x+\xx+2,\y+\yy+1) circle (0.25);
\draw[fill=black] (\x+\xx+1,\y+\yy+2) circle (0.25);
\draw[fill=black] (\x+\xx+0,\y+\yy+1) circle (0.25);
}
}
}
}

\foreach\x in {0,12}{
\foreach\y in {6}{

\foreach\xx in {0,2,4}{
\foreach\yy in {0,4}{
\draw(\x+\xx+0,\y+\yy+0)--(\x+\xx+2,\y+\yy+0)--(\x+\xx+2,\y+\yy+2)--(\x+\xx+0,\y+\yy+2)--cycle;

\draw[fill=black] (\x+\xx+0,\y+\yy+0) circle (0.25);
\draw[fill=black] (\x+\xx+0,\y+\yy+2) circle (0.25);
\draw[fill=black] (\x+\xx+2,\y+\yy+2) circle (0.25);
\draw[fill=black] (\x+\xx+2,\y+\yy+0) circle (0.25);

\draw[fill=black] (\x+\xx+1,\y+\yy+0) circle (0.25);
\draw[fill=black] (\x+\xx+2,\y+\yy+1) circle (0.25);
\draw[fill=black] (\x+\xx+1,\y+\yy+2) circle (0.25);
\draw[fill=black] (\x+\xx+0,\y+\yy+1) circle (0.25);
}
}

\foreach\xx in {0,4}{
\foreach\yy in {2}{
\draw(\x+\xx+0,\y+\yy+0)--(\x+\xx+2,\y+\yy+0)--(\x+\xx+2,\y+\yy+2)--(\x+\xx+0,\y+\yy+2)--cycle;

\draw[fill=black] (\x+\xx+0,\y+\yy+0) circle (0.25);
\draw[fill=black] (\x+\xx+0,\y+\yy+2) circle (0.25);
\draw[fill=black] (\x+\xx+2,\y+\yy+2) circle (0.25);
\draw[fill=black] (\x+\xx+2,\y+\yy+0) circle (0.25);

\draw[fill=black] (\x+\xx+1,\y+\yy+0) circle (0.25);
\draw[fill=black] (\x+\xx+2,\y+\yy+1) circle (0.25);
\draw[fill=black] (\x+\xx+1,\y+\yy+2) circle (0.25);
\draw[fill=black] (\x+\xx+0,\y+\yy+1) circle (0.25);
}
}
}
}

\end{tikzpicture}
\caption{The Sierpi\'nski carpet cable system}
\label{SC cable example}
\end{figure}
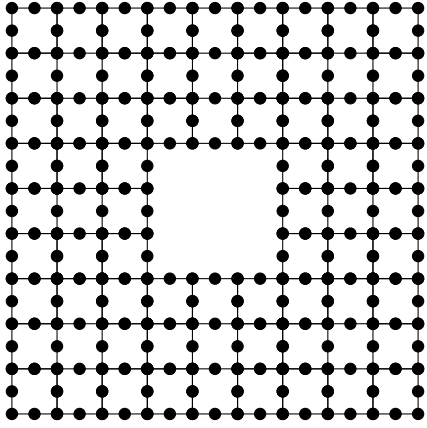

\end{example}

From now on we consider a local tree $(X,d)$ and always assume it is locally compact, connected and separable.  If $x,y \in X$ with $d(x,y) < \iota (x)$ we denote 
\[
[x,y]=\phi_{x,y} ([0, d(x, y)]) \quad \text{and} \quad  ]x,y[=\phi_{x,y} (]0, d(x, y)[).
\]

Given $u,v,w \in B(x,\iota(x))$, there exists a unique point $c(u,v,w) \in B(x,\iota(x))$ such that 
 \[
 [u,w]\cap[u,v]=[u,c(u,v,w)].
 \]
This point also satisfies $[v,u]\cap [v,w]=[v,c(u,v,w)]$ and $[w,u]\cap [w,v]=[w,c(u,v,w)]$ (see Lemma 1.2 in \cite{Chiswell} or Lemma 3.20 in \cite{Evans08}). 

\begin{figure}[ht]
\centering
\begin{tikzpicture}[scale=2]

\draw (0,0)--(0,1);
\draw (0,0)--(1/2*1.73205080756,-1/2);
\draw (0,0)--(-1/2*1.73205080756,-1/2);

\draw[fill=black] (0,0) circle (0.05);
\draw[fill=black] (0,1) circle (0.05);
\draw[fill=black] (1/2*1.73205080756,-1/2) circle (0.05);
\draw[fill=black] (-1/2*1.73205080756,-1/2) circle (0.05);

\draw (0.5,0.15) node {$c(u,v,w)$};
\draw (-1/2*1.73205080756,-1/2-0.2) node {$u$};
\draw (1/2*1.73205080756,-1/2-0.2) node {$v$};
\draw (0.2,1) node {$w$};
\end{tikzpicture}
\caption{Tripod defining $c(u,v,w)$}\label{fig_uvwc}
\end{figure}
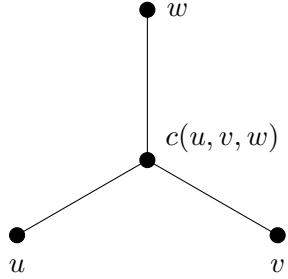

The set
\[
\mathcal{S}=\bigcup_{x \in X} \bigcup_{y \in B(x,\iota(x))}]x,y[  
\]
is called the skeleton of $X$.   If $\mathcal V$ be a countable dense subset of $\mathcal S$, one can also write
\[
\mathcal{S}=\bigcup_{x,y \in \mathcal V, d(x,y) < \iota (x)} ]x,y[.
\]
Notice that $\mathcal{S}$ is dense in $X$ but that $\mathcal S$ could be much smaller than $X$ itself as the example below of the Vicsek set shows.

Given a local tree, there exists a unique measure $\nu$ on the $\sigma$-field generated by $\{ ]x,y[, x,y \in \mathcal V \}$ such that for every $x,y \in \mathcal S$ with $d(x,y) < \iota(x)$,
\[
\nu ( ]x,y [)=d(x,y).
\]
The measure $\nu$ will be called the length measure. In general, $\nu$ is not a Radon measure: For instance, in the Vicsek set, the $\nu$-measure of any ball of positive radius is $\infty$ since any such ball contains a scaled copy of the Vicsek set and the skeleton of the Vicsek set has infinite length.

\begin{example}[Vicsek set]\label{Vicsek set}

\begin{figure}[htb]\label{figure1}
 \noindent
 \makebox[\textwidth]{\includegraphics[height=0.22\textwidth]{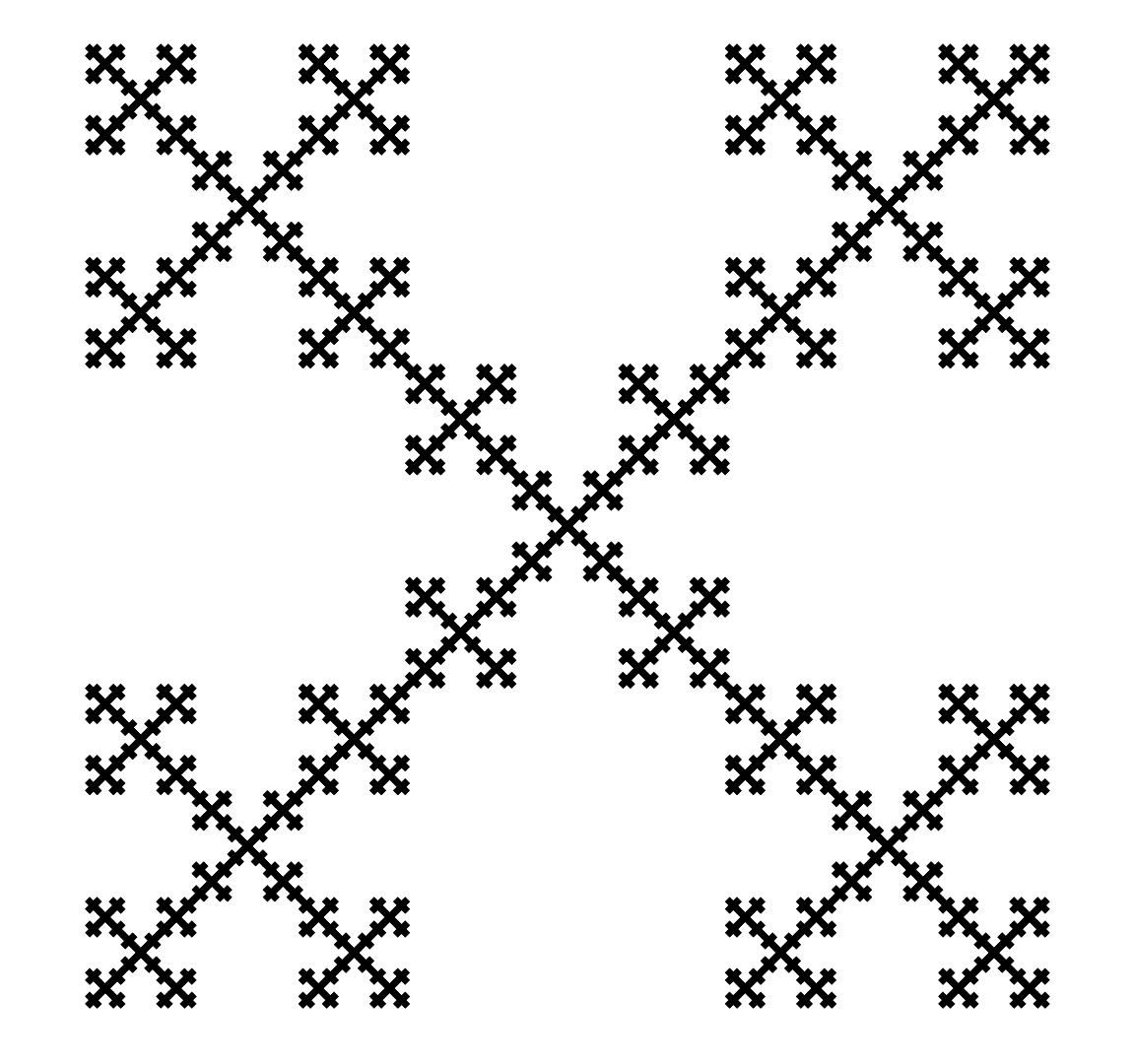}}
  	\caption{Vicsek set}
    \label{example Vicsek set}
\end{figure}

Let  $q_1=(0,0)$ be the center  of the unit square in the  plane  and let $q_2$, $q_3$ , $q_4$, and $q_5$ be its 4 corners arranged in such a way that $q_2$ and $q_4$ are diametrically opposed. Define $\psi_i(z)=\frac13(z-q_i)+q_i$ for $1\le i\le 5$. The Vicsek set $V$, see Figure \ref{example Vicsek set}, is the unique non-empty compact set in the plane such that 
\[
V=\bigcup_{i=1}^5 \psi_i(V).
\]
Equipped with the geodesic metric, $V$ is a metric tree and  its skeleton is given by
\[
\mathcal{S}=D \cup \bigcup_{k=1}^{\infty} \bigcup_{i_1,i_2,\cdots, i_k=1}^5  \psi_{i_1} \cdots \psi_{i_k} (D)
\]
where $D$ is the union of the two open diagonals of the square: $]q_2,q_4[ \cup ]q_3,q_5[$. Then, even though $\mathcal{S}$ is dense in $V$,  the metric tree $V$ has Hausdorff dimension $\frac{\ln 5}{\ln 3}$ whereas $\mathcal{S}$ supports the one-dimensional length measure. We refer to \cite{MR4494039} for details about the Vicsek set.
\end{example}

%
%
%
%

\begin{figure}[ht]
\begin{center}
\scalebox{0.5}{
\begin{tikzpicture}

\newcommand{\stepzero}[1]{
    \begin{scope}[shift={(#1)}]
     \foreach\ddd in {10pt}{
     
 \node at ($(#1)+(0*\ddd,0*\ddd)$) {\includegraphics[width=\ddd]{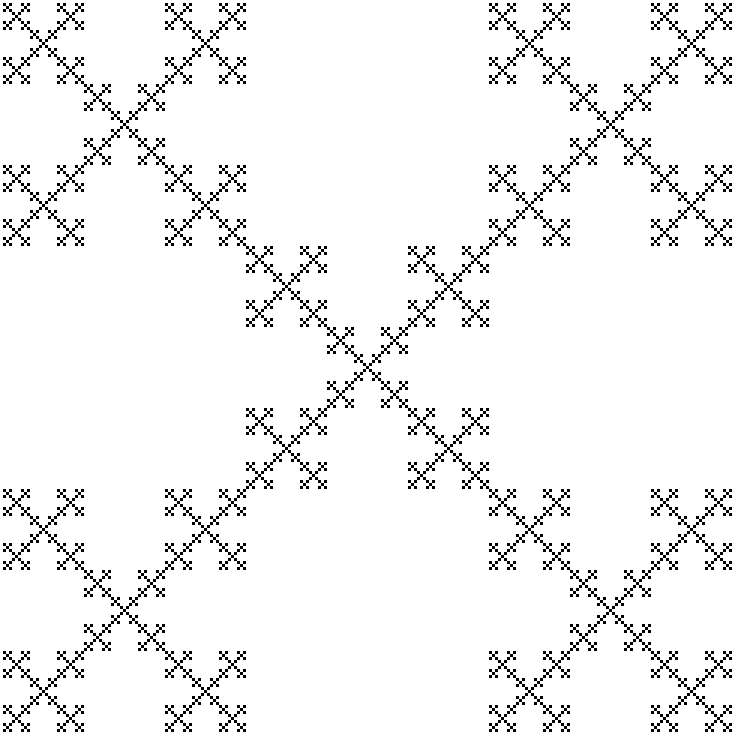}};
 \node at ($(#1)+(0*\ddd,2*\ddd)$) {\includegraphics[width=\ddd]{vc}};
 \node at ($(#1)+(0*\ddd,4*\ddd)$) {\includegraphics[width=\ddd]{vc}};
 \node at ($(#1)+(1*\ddd,1*\ddd)$) {\includegraphics[width=\ddd]{vc}};
 \node at ($(#1)+(1*\ddd,3*\ddd)$) {\includegraphics[width=\ddd]{vc}};
 \node at ($(#1)+(2*\ddd,0*\ddd)$) {\includegraphics[width=\ddd]{vc}};
 \node at ($(#1)+(2*\ddd,2*\ddd)$) {\includegraphics[width=\ddd]{vc}};
 \node at ($(#1)+(2*\ddd,4*\ddd)$) {\includegraphics[width=\ddd]{vc}};
 \node at ($(#1)+(3*\ddd,1*\ddd)$) {\includegraphics[width=\ddd]{vc}};
 \node at ($(#1)+(3*\ddd,3*\ddd)$) {\includegraphics[width=\ddd]{vc}};
 \node at ($(#1)+(4*\ddd,0*\ddd)$) {\includegraphics[width=\ddd]{vc}};
 \node at ($(#1)+(4*\ddd,2*\ddd)$) {\includegraphics[width=\ddd]{vc}};
 \node at ($(#1)+(4*\ddd,4*\ddd)$) {\includegraphics[width=\ddd]{vc}};
        }
    \end{scope}
}

\foreach\eee in {25pt}
{
\stepzero{0,0};
\stepzero{0*\eee,2*\eee};
\stepzero{0*\eee,4*\eee};
\stepzero{1*\eee,1*\eee};
\stepzero{1*\eee,3*\eee};
\stepzero{2*\eee,0*\eee};
\stepzero{2*\eee,2*\eee};
\stepzero{2*\eee,4*\eee};
\stepzero{3*\eee,1*\eee};
\stepzero{3*\eee,3*\eee};
\stepzero{4*\eee,0*\eee};
\stepzero{4*\eee,2*\eee};
\stepzero{4*\eee,4*\eee};
}
\end{tikzpicture}
}
\end{center}
\caption{Example of a local tree: Blow up of  the Vicsek set by another fractal}\label{fig_VCFC}

\end{figure}
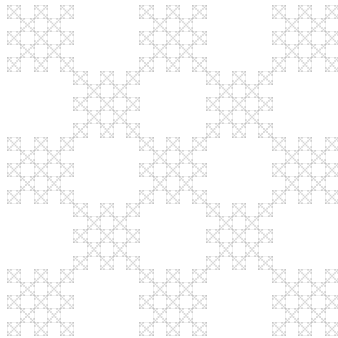

\begin{figure}[htb]
 \noindent
 \makebox[\textwidth]{\includegraphics[height=0.22\textwidth]{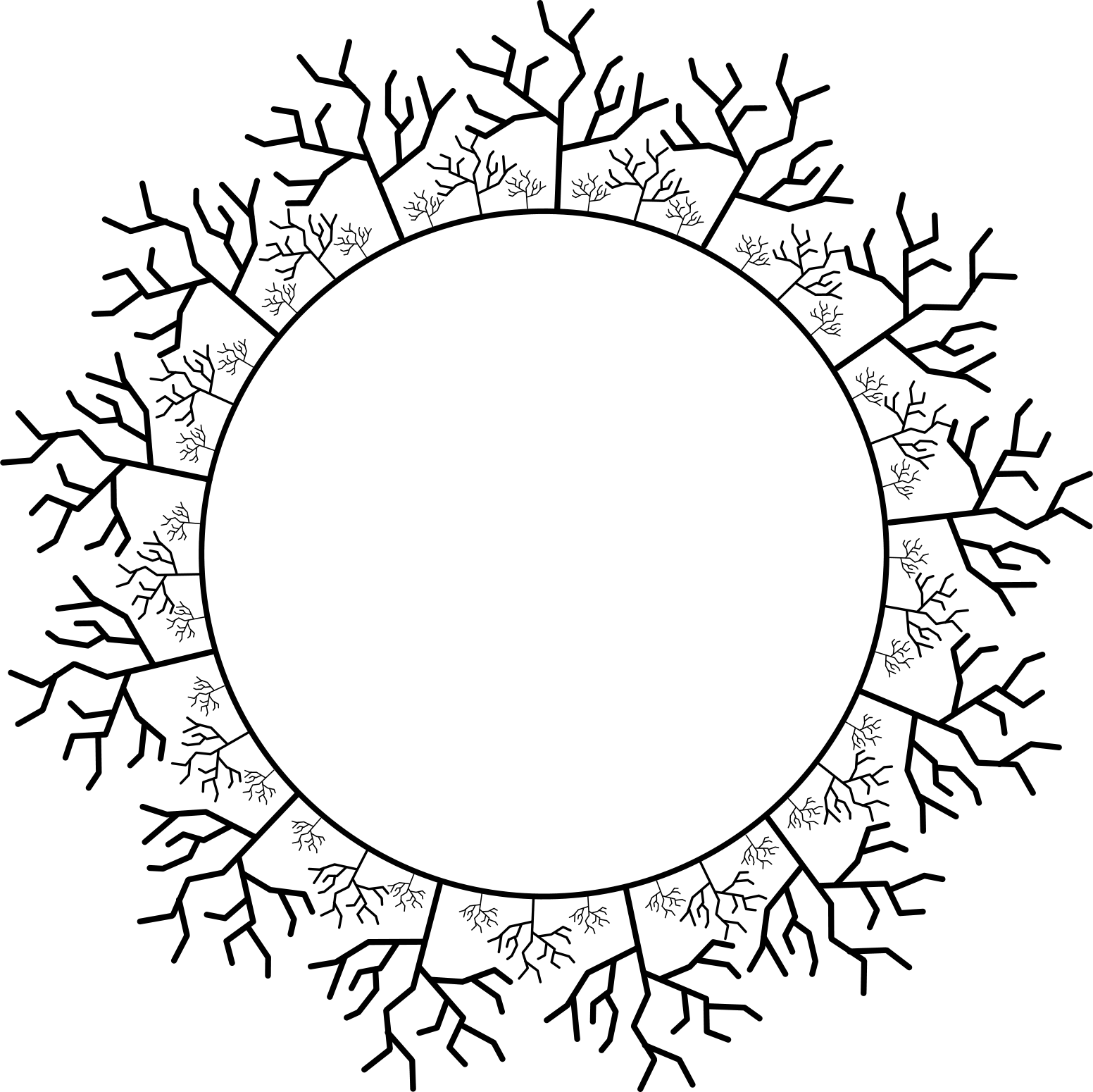}}
  	\caption{Example of a local tree}
    \label{berkovich}
\end{figure}

\begin{example}
The Figures \ref{fig_VCFC} and \ref{berkovich} illustrate further examples of local trees. The Figure \ref{fig_VCFC} is the blow up of the Vicsek of tree by another fractal, see \cite[Section 9]{gao2024holderregularityharmonicfunctions} and Figure \ref{berkovich} is the Berkovich space of a Tate elliptic curve.
\end{example}

%

\subsection{Weak  gradients}
The notion of weak gradient is crucial in our analysis. In the framework of trees it appears in \cite{AEW}.

\begin{definition}
    Let $f \in C(X)$ and $x \in X$. We say that $f$ is absolutely continuous in $B(x ,\iota(x))$ if there exists a $\nu$-measurable function $g$ on $\mathcal{S}\cap B(x,\iota(x))$ such that for every $u\in B(x,\iota(x))$, $g$ is $\nu$-integrable on $]x,u[$ and 
\begin{equation}\label{eq:weakGradient1}
f(u)-f(x)=\int_{]x,u[} g \, d\nu.
\end{equation}
If it  exists, the function $g$ is unique (up to a $\nu$ null set) and is called the weak gradient of $f$ in $B(x,\iota(x))$. We will denote $g=\partial_x f$.
\end{definition}

Note that if $f$ is absolutely continuous in $B(x ,\iota(x))$ then for every $u,v\in B(x,\iota(x))$ with $v \in [x , u]$
\begin{equation}\label{eq:weakGradient2}
f(u)-f(v)=\int_{]v,u[} \partial_x f \, d\nu.
\end{equation}
Indeed, one has
\[
f(v)-f(x)=\int_{]x,v[} \partial_x f \, d\nu, \quad f(u)-f(x)=\int_{]x,u[} \partial_x f \, d\nu,
\]
so that
\[
f(u)-f(v)=\int_{]x,u[} \partial_x f \, d\nu -\int_{]x,v[} \partial_x f \, d\nu=\int_{]v,u[} \partial_x f \, d\nu.
\]
\begin{definition}
Let $f \in C(X)$. We say that $f$ is absolutely continuous and write $f \in AC(X)$ if it is  absolutely continuous in $B(x ,\iota(x))$ for all $x \in X$.
\end{definition}

If $f \in AC(X)$ is absolutely continuous and $x,y \in X$ with $x\neq y$ then it is clear that $| \partial_x f |=|\partial_y f|$  holds $\nu$-a.e.  
 on $\mathcal{S} \cap B(x ,\iota(x))\cap B(y ,\iota(y))$. We will therefore simply use the notation $| \partial f|$. Similarly, if $f,g \in AC(X)$ then $\partial_x f \partial_x g$ does not depend on $x$ and we will denote $\partial f \partial g$.  We also  observe the following:
 
\begin{lemma}\label{absolute continuity}
Let $f \in AC(X)$ and $x \in X$, then for every $u,v \in B(x,\iota(x))$
\begin{align}\label{domination}
| f(v) -f(u) | \le \int_{]u,v[} |\partial f| \, d\nu.
\end{align}
\end{lemma}

\begin{proof}
From the previous remark one has,
\[
f(u)-f(c(u,v,x))=\int_{]c(u,v,x),u[} \partial_x f d\nu, \quad f(v)-f(c(u,v,x))=\int_{]c(u,v,x),v[} \partial_x f d\nu.
\]
Therefore,
\begin{align*}
|f(v)-f(u)|& =\left| \int_{]c(u,v,x),v[} \partial_x f d\nu- \int_{]c(u,v,x),u[} \partial_x f d\nu \right| \\
 &\le  \int_{]c(u,v,x),v[} | \partial f| d\nu+ \int_{]c(u,v,x),u[} |\partial_x f | d\nu \\
 &=\int_{]u,v[} |\partial f| \, d\nu.
\end{align*}
\end{proof}

The following lemma is easy to prove due to the corresponding properties that hold for absolutely continuous functions on the line.
 
\begin{lemma}[Chain and Leibniz rules]\label{Chain rule}
\

\begin{enumerate}
 \item Let $f \in AC(X)$ and let $\Phi$ be a $C^1$ function on $\mathbb R$. Then, $\Phi(f) \in AC(X)$ and $\nu$ a.e.
 \[
 |\partial \Phi(f)|=|\Phi'(f)| \, |\partial f|.
 \]
 \item Let $f,g \in AC(X)$. Then $fg\in AC(X)$  and $\nu$ a.e.
 \[
 |\partial (fg)| \le | f | \, |\partial g| +|g| \, | \partial f|.
 \]
\end{enumerate}
\end{lemma}

\begin{definition}
Let $1 \le p\le \infty$. For $f\in C(X)$, we say that $f\in AC^{p}(X)$ if it admits a weak gradient  $\partial f\in L^p(\mathcal S, \nu)$. 
\end{definition}

Functions in $AC^{p}(X)$ satisfy a local Morrey estimate.

\begin{theorem}[Local Morrey type estimate]\label{thm:Morrey}
 Let $1 \le p \le \infty$ and $x_0 \in X$. 
 \begin{enumerate}
\item If $p \ge 1$, then for every $f \in AC^{p}(X)$ and $x,y \in B(x_0,\iota(x_0))$
\[
|f(x)-f(y)|^p \le d(x,y)^{p-1} \int_{]x,y[} |\partial f |^p d\nu\le d(x,y)^{p-1} \|\partial f\|^p_{L^p(\CS,\nu)}.
\]
\item If $p =+\infty$, then for every $f \in AC^{\infty}(X)$ and $x,y \in B(x_0,\iota(x_0))$
\[
|f(x)-f(y)| \le d(x,y)  \|\partial f\|_{L^\infty(\CS,\nu)}.
\]

\end{enumerate}
Therefore for $p>1$, any $f \in AC^{p}(X)$ is locally $1-\frac{1}{p}$ H\"older continuous.
\end{theorem}
\begin{proof}

We note that it  follows from Lemma \ref{absolute continuity} that if $f\in AC^{p}(X)$, then for every $x_0 \in X$, $x,y \in B(x_0,\iota(x_0))$
\[
| f(x) -f(y) | \le \int_{]x,y[} | \partial f | d\nu,
\]
and the conclusion follows from H\"older's inequality.
\end{proof}

\subsection{Sobolev, BV spaces and $p$-energy forms}

We now consider a Radon measure $m$ on $(X,d)$ with full support. Typically, in many situations $m$ and $\nu$ are singular with respect to each other and we even have $m(\mathcal{S})=0$; see the Vicsek set Example \ref{Vicsek set} for which we take for $m$ the Hausdorff measure.

\begin{definition}
Let $1 \le p\le \infty$.  For $f \in L^p(X,m)$ we say that $f\in W^{1,p}(X,m)$ if there exists a $m$-representative $\tilde{f}$ of $f$ such that $\tilde{f} \in  AC^{p}(X)$. In that case, we will denote $|\partial f| :=|\partial \tilde{f}|$. The norm on $W^{1,p}(X,m)$ is defined as
\[
\| f \|_{W^{1,p}(X,m)} =\| f \|_{L^p(X,m)}+\|\partial f\|_{L^p(\CS,\nu)}.
\]
\end{definition}

Since for  $f \in W^{1,p}(X,m)$ there exists a (necessarily unique) $m$-representative which is continuous, we will often simply work with this representative without further notice so that $f(x)$ is well defined for a given $x \in X$.  We first collect some basic properties that will be important in the sequel.

\begin{proposition}\label{closedness p-energy}
Let $1 \le p \le \infty$. The space  $(W^{1,p}(X,m),\| \cdot \|_{W^{1,p}(X,m)})$ is a Banach space.
\end{proposition}

\begin{proof}
Let $f_n$ be a Cauchy sequence in $W^{1,p}(X,m)$.  Since  $f_n$ is then also a Cauchy sequence in $L^p(X,m)$ we have   $f_n \to f$ in $L^p(X,m)$ where $f \in L^p(X,m)$. Let now $x_0 \in X$. Since $\partial_{x_0} f_n $ is a Cauchy sequence in $L^p( \mathcal{S} \cap B(x_0,\iota(x_0)),\nu)$ we  have that $\partial_{x_0} f_n $ converges to some $g$ in $L^p( \mathcal{S} \cap B(x_0,\iota(x_0)),\nu)$. We can extract from $f_n$ a subsequence $f_{n_k}$ that converges $m$-a.e. to  $f$. One has for every $x,y \in B(x_0,\iota(x_0))$,
\[
f_n(y)-f_n(x)= \int_{]x_0,y[} \partial_{x_0} f_n d\nu-\int_{]x_0,x[} \partial_{x_0} f_n d\nu.
\]
Taking limit along the subsequence yields that for $m$-almost every  $x,y \in B(x_0,\iota(x_0))$,
\[
f(y)-f(x)=\int_{]x_0,y[} g d\nu-\int_{]x_0,x[} g d\nu.
\]
The function $y \to \int_{]x_0,y[} g d\nu$ is continuous on $B(x_0,\iota(x_0))$. Therefore, $f$ admits on $B(x_0,\iota(x_0))$ a continuous $m$-representative $\tilde{f}$ that satisfies for every  $y \in B(x_0,\iota(x_0))$,
\[
\tilde{f}(y)-\tilde{f}(x_0)=\int_{]x_0,y[} g d\nu.
\]
Since $m$ has full support, the continuous representatives obtained on two different balls coincide on their intersection. Therefore $\tilde{f} \in AC(X)$ and $g=\partial_{x_0} \tilde{f}$. If $p<\infty$, then from Fatou's lemma  we have $|\partial \tilde{f}| \in L^p(\mathcal{S},\nu)$, which implies $f \in W^{1,p}(X,m)$. If $p=\infty$, then we have $\| \partial \tilde{f} \|_{L^\infty(\mathcal{S},\nu)} \le \sup_n  \| \partial f_n \|_{L^\infty(\mathcal{S},\nu)}$ which implies $f \in W^{1,\infty}(X,m)$.

 It remains to prove that $f_n \to f$ in $W^{1,p}(X,m)$. We assume $p<\infty$, the case $p=\infty$ actually being simpler. Let $\varepsilon >0$. Since $f_n$ is a Cauchy sequence in $W^{1,p}(X,m)$, one has for  $n,m \ge N$, $\int_{\mathcal{S}} | \partial f_n -\partial f_m |^p d\nu \le \varepsilon$, where $N$ is large enough. Letting $n \to \infty$ along a subsequence yields by Fatou's lemma $\int_{\mathcal{S}} | \partial f -\partial f_m |^p d\nu \le \varepsilon$ and the conclusion follows.
\end{proof}

\begin{proposition}
    Let $1<p<+\infty$. The Banach space $(W^{1,p}(X,m),\| \cdot \|_{W^{1,p}(X,m)})$ is reflexive and separable. 

\end{proposition}

\begin{proof}
To prove reflexivity, we prove that $(W^{1,p}(X,m), ||| \cdot ||| _{W^{1,p}(X,m)})$ is uniformly convex, where $||| \cdot |||_{W^{1,p}(X,m)}$ is the equivalent norm defined by
\[
||| f |||_{W^{1,p}(X,m)}=\left( \| f \|^p_{L^p(X,m)}+\|\partial f\|^p_{L^p(\CS,\nu)}\right)^{1/p}
\]

Actually, we easily get the following Clarkson inequalities which imply uniform convexity. Let $f,g\in W^{1,p}(X,m)$, $1<p<\infty$, and $q$ be the H\"older conjugate of $p$. If $2\le p<\infty$, then
\begin{equation}\label{eq:CTIge2}
||| (f+g)/2|||_{W^{1,p}(X,m)}^p + ||| (f-g)/2 |||_{W^{1,p}(X,m)}^p
\le ||| f |||_{W^{1,p}(X,m)}^p/2 +||| g |||_{W^{1,p}(X,m)}^p/2.
\end{equation}
If $1<p\le 2$, then
\begin{equation}\label{eq:CTIle2}
||| (f+g)/2 ||| _{W^{1,p}(X,m)}^{q} +  ||| (f-g)/2 |||_{W^{1,p}(X,m)}^{q}
\le \brak{ ||| f ||| _{W^{1,p}(X,m)}^p/2 + ||| g |||_{W^{1,p}(X,m)}^p/2}^{q-1}.
\end{equation}
It remains to prove separability. The identity map $\iota: (W^{1,p}(X,m), \| \cdot \|_{W^{1,p}(X,m)} \to (L^{p}(X,m),  \| \cdot \|_{L^{p}(X,m)})$ is a linear and bounded injective map. Since the space $(W^{1,p}(X,m), \| \cdot \|_{W^{1,p}(X,m)} $ is reflexive and $L^{p}(X,m)$ is separable because $X$ is, it now follows from Proposition 4.1 in \cite{alvarado2023simple} that  $(W^{1,p}(X,m), \| \cdot \|_{W^{1,p}(X,m)}) $ is separable.
\end{proof}

\begin{corollary}\label{lower semi-continuity p-energy}
   Let $p>1$. The $p$-energy form 
\[
\mathcal{E}_p(f):=\int_{\mathcal{S}} | \partial f |^p d\nu
\]
 with domain $W^{1,p}(X,m)$ is lower semi-continuous in $L^p(X,m)$.
\end{corollary}

\begin{proof}
Let $f_n$ be a sequence in $W^{1,p}(X,m)$ such that $f_n \to f$ in $L^p(X,m)$ and $\mathcal{E}_p(f_n)$ is convergent. Since the Banach space $(W^{1,p}(X,m),\| \cdot \|_{W^{1,p}(X,m)})$ is reflexive one can extract a subsequence such that $f_{n_k}$ is weakly convergent to some $g \in W^{1,p}(X,m)$. Mazur's lemma implies that a convex combination of the $f_{n_k}$'s converges in $W^{1,p}(X,m)$ to $g$. Since $f_{n_k}$ converges in $L^p(X,m)$ to $f$ one must have $g=f$. Therefore, $f \in W^{1,p}(X,m)$. Moreover, by convexity of the functional $\mathcal{E}_p$ one has
\[
\mathcal{E}_p(f) \le \lim_{n \to +\infty} \mathcal{E}_p(f_n).
\]
\end{proof}

The $1$-energy is not lower semicontinuous in general.  To get a satisfying theory one can use $L^1$-relaxation of the $1$-energy as in \cite{Mir03}. This leads to the notion of bounded variation function (BV).

\begin{definition}
The space of bounded variation functions is defined by $L^1$ relaxation:
\[
BV(X,m) =\left\{ f \in L^1(X,m) \mid \| \partial f \|(X)<+\infty \right\}
\]
where
\[
 \| \partial f \|(X)=\inf_{f_n} \liminf_{n \to +\infty} \int_\mathcal{S} | \partial f_n | d\nu
\]
and the $\inf$ is taken over the set of functions $f_n \in W^{1,1}(X,m)$ such that $f_n \to f$ in $L^1(X,m)$. In other words $\| \partial f \|(X)$ is the lower semi-continuous envelope in $L^1(X,m)$ of the functional $\int_\mathcal{S} |\partial f| d\nu$. The norm on $BV(X)$ is defined by
\[
\| f \|_{BV(X)} =\| f \|_{L^1(X,m)} +  \| \partial f \|(X).
\]
\end{definition}

\begin{lemma}
We have $W^{1,1}(X,m) \subset BV(X,m)$ and for $f \in W^{1,1}(X,m)$,
\[
\| \partial f  \|(X)=\int_\mathcal{S} | \partial f | d\nu.
\]
\end{lemma}

\begin{proof}
It is plain that $W^{1,1}(X,m) \subset BV(X,m)$ and that for $f \in W^{1,1}(X,m)$,
\[
\| \partial f  \|(X) \le \int_\mathcal{S} | \partial f | d\nu.
\]
Consider now $f \in W^{1,1}(X,m)$ and  $f_n \in W^{1,1}(X,m)$ such that $f_n \to f$ in $L^1(X,m)$ and $\int_\mathcal{S} | \partial f_n | d\nu \to C$ with $C \ge 0$. One can extract a subsequence $f_{n_k}$ that pointwisely converges to $f$ on a set $E \subset X$ such that $m(X \setminus E)=0$. Since $m$ has full support, the set $E$ is dense in $X$. Let $x \in X$. We have then for every $u,v \in B(x,\iota(x))$,
\[
| f_n(u)-f_n(v)| \le \int_{]u,v[} | \partial f_n | d\nu
\]
which yields for every $u,v \in E \cap B(x,\iota(x))$,
\[
| f(u)-f(v)| \le \liminf \int_{]u,v[} | \partial f_{n_k} | d\nu.
\]
Let now $\mathcal{A}$ be any finite union of sets of the type $]u,v[$ with $d(u,v) < \iota(u)$ and let $[u_i,v_i]$, $1 \le i \le N$, be pairwise disjoint subsets of $\mathcal{A}$ with $d(u_i,v_i) < \iota(u_i)$. Since $E$ is dense in $X$, one can find sequences $u_i^j, v_i^j \in E \cap B(u_i,\iota(u_i))$ such that $u_i^j \to u_i$ and $v_i^j \to v_i$ when $j \to +\infty$. Since $]u_i^j,v_i^j[ \subset [u_i^j,u_i] \cup [u_i,v_i] \cup [v_i,v_i^j]$, the sets $]u_i^j,v_i^j[$, $1 \le i \le N$, are pairwise disjoint for $j$ large enough. Summing up the inequalities
\[
| f(u_i^j)-f(v_i^j)| \le \liminf \int_{]u_i^j,v_i^j[} | \partial f_{n_k} | d\nu
\]
then yields $\sum_{i=1}^N | f(u_i^j)-f(v_i^j)| \le \liminf \int_{\mathcal{S}} | \partial f_{n_k} | d\nu = C$. Letting $j \to +\infty$, we obtain from the continuity of $f$ that $\sum_{i=1}^N | f(u_i)-f(v_i)| \le C$ and taking the  supremum over such families of subsets of $\mathcal{A}$ yields
\[
\int_{\mathcal{A}} | \partial f | d\nu \le C.
\]
This gives
\[
\int_{\mathcal{S}} | \partial f | d\nu \le C
\]
and therefore
\[
\int_\mathcal{S} | \partial f | d\nu \le \| \partial f  \|(X) .
\]
\end{proof}

We will see in next section that under some appropriate assumptions $BV(X,m)$ is a Banach space, see Corollary \ref{Banach space property}. We note that, in general, functions in $BV(X,m)$ are not continuous; for instance if $X=\mathbb{R}$ and $m=\nu=\textrm{Lebesgue measure}$, then $1_{[0,1]} \in BV(X,m)$.  More generally, assume  that $X$ is a cable system as in Example \ref{cable system example}.  Let $a,b\in X$ with $d(a,b) < \min (\iota (a),\iota (b))$ be such that $]a,b[$ does not contain any vertex. Assume that $N_a$ (resp. $N_b$), the number  of connected components of $B(a,d(a,b))\setminus \{ a\}$ (resp. $B(b,d(a,b))\setminus \{ b\}$), is finite, then $1_{[a,b]} \in BV(X,m)$ and 
\[
\| \partial 1_{[a,b]} \| (X) \le N_a+N_b-2.
\]
Indeed, in that case it is easy to check that the sequence 
\[
f_n(x) =\left(1-nd(x,c(a,b,x)) \right)_+
\]
(defined for $x \in B(a,\iota(a))$, $n$ large enough, and extended by $0$ outside of $B(a,\iota(a))$) converges to $1_{[a,b]} $ in $L^1(X,m)$ and that $\int_\mathcal{S} | \partial f_n | d\nu \le  N_a+N_b-2$.

\section{Korevaar-Schoen type characterizations of the Sobolev and BV spaces and applications}\label{KS section}

We first introduce the two main  assumptions we will be working with in this section. Let $r_0>0$ and   $\Phi:[0,r_0] \to[0,+\infty)$ be an increasing  function such that there exist $ \alpha \ge 1$ such that

\begin{align}\label{volume assumption}
c\frac{R}{r}\le\frac{\Phi(R)}{\Phi(r)} \le C \left(\frac{R}{r}\right)^\alpha \text{ for any }r\le R \le r_0.
\end{align}

We make the assumption of \textbf{uniform volume control}:

\

\noindent 

\textit{There exists $r_0>0$ such that for every $x \in X$ and $0<r<r_0$}
$$c \Phi(r)\le m(B(x,r))\le C\Phi(r).$$

\

In particular, the measure $m$ is uniformly locally doubling, i.e. for every $x \in X$ and $0<r<r_0/2$
\[
 m(B(x,2r))\le C m(B(x,r)),
\]
and as a consequence, see \cite[Lemma B.3.4.]{Gromov}, the metric space $(X,d)$ is uniformly locally doubling, which means: There exist $R>0$ and $N=N_R$ such that for every $x \in X, 0<r<R$, any ball $B(x,r)$ can be covered by at most $N$ balls of radius $r/2$ where $N$ is independent of $x$ and $r$. 

\

The second main assumption of this section is the following: 

\

\noindent \textit{The metric space $(X,d)$ is a \textbf{uniform local tree} meaning that there exists $0< \underline{\iota} \le \mathrm{diam}(X)$ such that for every $x \in X$, $B(x,\underline{\iota})$ is a real tree.}

\
 
\begin{Example}
   Let $(\mathbf{G},d)$ be a cable system as in Example  \ref{cable system example}. Assume such that for some $N \ge 1$ and $a >0$:
   \begin{enumerate}
\item Any vertex has at most $N$ neighboring vertices and is the endpoint of at most $N$ rays;
\item Any edge has length more than $a>0$;
   \end{enumerate}
   Then $(\mathbf{G},d,m)$ satisfies  the assumptions with $m$ being the length measure. In that case, one can take $r_0=a$,  $\Phi(r)=r$  and $\underline{\iota}=a/2$. 
\end{Example}

\begin{Example}[Vicsek sets]
The Vicsek set $(V,d,m)$ from Example \ref{Vicsek set} satisfies the assumptions with $m$ being the Hausdorff measure. In that case one can take $r_0=\mathrm{diam}(V)$,  $\Phi(r)=r^{\frac{\ln 5}{\ln 3}}$  and $\underline{\iota}=\mathrm{diam}(V)$. More generally the scale irregular Vicsek sets introduced in  \cite{chen2024besovlipschitznormpenergymeasure} satisfy the assumptions. Interestingly, for those scale irregular Vicsek sets, the function $\Phi$ has several scales, see  \cite[Proposition 2.9]{chen2024besovlipschitznormpenergymeasure}.
\end{Example}

\subsection{Controlled partitions of unity}

Our key lemma is the following.

\begin{lemma}\label{bump function}
Let $\varepsilon_0>0$ be small enough.  Let $x \in X$ and $0< \varepsilon < \varepsilon_0$. There exists an absolutely continuous function $\Psi^{\varepsilon}_x \in AC(X)$  such that:

\begin{enumerate}
\item $\Psi^{\varepsilon}_x(x)=1$;
\item $0\le \Psi^{\varepsilon}_x \le 1$;
\item If $d(y,x) \le \varepsilon/2$ then $\Psi^{\varepsilon}_x (y) \ge 1/2$;
\item If $d(y,x) \ge 2 \varepsilon $ then  $\Psi^{\varepsilon}_x(y)=0$;
\item For $\nu$-a.e.  $y \in \mathcal{S}$,  $|\partial \Psi^{\varepsilon}_x (y) | \le \frac{1}{\varepsilon}$ and if $d(x,y) \ge \varepsilon$, $|\partial \Psi^{\varepsilon}_x (y) |=0$;
\item  $\nu \left( \mathrm{supp} (\partial \Psi^{\varepsilon}_x) \right) \le C \varepsilon$ where the constant $C>0$ is independent from $x$ and $\varepsilon$.
\end{enumerate}
\end{lemma}

\begin{proof}
Let $x \in X$ and $\varepsilon_0>0$ be small enough. For $\varepsilon < \varepsilon_0$, consider the ball $B(x,3\varepsilon)$ which is therefore a  tree since $3\varepsilon < \underline{\iota}$. Denote $R_\varepsilon=\{ u \in B(x,3\varepsilon) \mid d(x,u)=\varepsilon \}$. For $u \in R_\varepsilon$, only one of the connected components of the set $B(x,3\varepsilon)\setminus \{ u\}$ contains $x$, we denote this component by $C^0_u$ and the other ones (if any) by $C^1_u,C^2_u,\cdots$. We consider then the set
\[
S_\varepsilon=\left\{ u  \in R_\varepsilon \mid \, \exists \,  i \ge 1,   \mathrm{diam} (C^i_u)  \ge \varepsilon \right\}.
\]
Note that, $X$ being connected, the set $S_\varepsilon$ is not empty if $\varepsilon_0 < \mathrm{diam}(X)/4$. Now, recall that as a consequence of the measure being locally doubling, the metric space $(X,d)$ is a uniformly locally doubling metric space. Therefore $B(x,3\varepsilon)$ can be always covered by at most $N$ balls of radius $\varepsilon/2$ where $N$ is independent of $x$ and $\varepsilon$. In particular, since  the $C^i_u$'s, $u \in R_\varepsilon$, $i \ge 1$, are  pairwise disjoint and since $C^i_u$ contains a point at distance $\varepsilon/2$ from $u$ when $\mathrm{diam} (C^i_u)  \ge \varepsilon$, we deduce that the number of elements in $S_\varepsilon$ is uniformly bounded above by a constant which is  independent from $x$ and $\varepsilon$.
We now define the function  $\Psi^{\varepsilon}_x$ as
\begin{align*}
\Psi^{\varepsilon}_x (y)=
\begin{cases}
1- \frac{1}{\varepsilon}\max_{u \in S_\varepsilon} d(x,c(x,y,u)), \quad y \in B(x,3\varepsilon) \\
0, \quad y \notin B(x,3\varepsilon).
\end{cases}
\end{align*}
Notice that for $y \in \cup_{u \in S_\varepsilon} [ x, u]$, one has 
\[
\Psi^{\varepsilon}_x (y)=1-\frac{1}{\varepsilon} d(x,y).
\]
We now check that $\Psi^{\varepsilon}_x $ satisfies the required conditions. (i) is clear and (ii) follows from the fact that for every $u \in S_\varepsilon$ and $y \in X$, $d(x,c(x,y,u)) \le d(x,u) \le \varepsilon$. Let now $y \in X$ with $d(x,y) \le \varepsilon/2$. We have for every $u \in S_\varepsilon$, $d(x,c(x,y,u)) \le d(x,y) \le \varepsilon/2$ so that $1- \frac{1}{\varepsilon}\max_{u \in S_\varepsilon} d(x,c(x,y,u)) \ge 1/2$ and (iii) is satisfied. Consider then $y \in B(x,3\varepsilon)$ with $d(x,y) \ge 2\varepsilon$. There exists some $u \in B(x,3\varepsilon)$ such that $d(x,u)=\varepsilon$ and $u \in [x,y]$. Since $d(u,y) \ge \varepsilon$, we have $u \in S_\varepsilon$. (iv) is therefore satisfied. Concerning (v) and (vi), we can just observe that $\nu$ a.e.
\[
\partial_x \Psi^{\varepsilon}_x =-\frac{1}{\varepsilon} 1_{\cup_{u \in S_\varepsilon} [ x, u]},
\]
so that the conclusion follows because the number of elements in $S_\varepsilon$ is uniformly bounded above.
\end{proof}

The next lemma ensures the existence of good coverings.

 \begin{lemma}\label{cover:local}
  Let $\lambda \ge 1$. Then, there exist  constants $C,r_0>0$ such that for every $r \in (0,r_0)$ the following holds: There exists an at most countable cover of $X$ made of balls $B_i$ of radius $r$ such that each point of $X$ belongs to at most $C$ of the balls $\lambda B_i$, i.e. $\sum_i 1_{\lambda B_i} \le C$.
 \end{lemma}

 \begin{proof}
The type of argument is well known, see, for instance, \cite[Lemma B.7.3]{Gromov}; We reproduce it in our setting. Let $x_n$ be a sequence such that $\{ x_n, n \in \mathbb{N}\}$ is dense in $X$. Let $0<r<\frac{R}{2\lambda}$ where $R$ is  such that for every $x \in X, 0<r<R$, any ball $B(x,r)$ can be covered by at most $N$ balls of radius $r/2$ where $N$ is independent of $x$ and $r$. Define inductively an at most countable set $(y_n)$ by defining $y_0:=x_0$ and $y_n:=x_k$ where $k$ is the least index $i$ such that $x_i \notin \cup_{j<n} B(y_j,r)$. If such an index does not exist, we do not define $y_n$. From the definition and the density of $(x_n)$, 
\[
X \subset \bigcup_{i} B( y_i,2r).
\]
Note also that for $i \neq j$, $d(y_i,y_j) \ge r$. We claim that each point of $X$ belongs to at most $C$ balls $B( y_i,2\lambda r)$ where $C$ is uniform. Indeed, for $x \in X$ let
\[
F_x=\{ y_i \mid x \in B(y_i,2\lambda r) \}
\]
Then $F_x \subset B(x,2\lambda r)$. Moreover,  $B(x,2\lambda r)$ can be covered by at most $C$ balls of radius $r/2$ and each of this ball can not contain more than one element of $F_x$. Therefore, the number of elements in $F_x$ is less than $C$.
 \end{proof}

With those two lemmas in hands we can construct convenient partitions of unity.

\begin{theorem}\label{partition unity}
Let $p \ge 1$, $\lambda \ge 4$ and let $\eps_0>0$ be small enough. For $0<\varepsilon<\varepsilon_0$, there exists an at most countable cover of $X$ made of balls $B^\eps_i=B(x_i,\eps)$ of radius $\eps$ such that each point of $X$ belongs to at most $C$ of the balls $ \lambda B^\eps_i$ and there exists and  a subordinated family of functions $\pip_i^\eps$ such that:
  \begin{enumerate}
 \item $\pip_i^\eps \in W^{1,p}(X,m)$;
\item $0\le \pip_i^\eps\le 1$ on $X$;
\item $\sum_i\pip_i^\eps=1$ on $X$; 
\item $\pip_i^\eps=0$ in $X\setminus B_i^{4 \,\eps}$, where $B_i^{4 \,\eps}=B(x_i,4\eps)$;
\item $ \int_\mathcal{S} |\partial \pip_i^\eps |^p d\nu  \le  \frac{C}{\eps^{ p-1}}$.
\end{enumerate}
\end{theorem}

\begin{proof}
For $0<\varepsilon<\varepsilon_0$ we first  consider an at most countable cover of $X$ made of balls $B^\eps_i=B(x_i,\eps)$ of radius $\eps$ as in Lemma \ref{cover:local} and define the functions  $\Psi^{2\varepsilon}_{x_i}$ as in Lemma \ref{bump function}. Since $B^\eps_i$ is a covering of $X$, any $x \in X$ belongs to at least one of the balls $B_i^\eps$ and at to  most $N$ of the balls $4B_i^\eps$ where is $N$ uniform, thus 
\[
 1/2 \le  \sum_j \Psi^{2\varepsilon}_{x_j} \le N
\]
Define then
\[
\pip_i^\eps=\frac{\Psi^{2\varepsilon}_{x_i}}{\sum_j \Psi^{2\varepsilon}_{x_j}}.
\]
It is immediate to check that (i)-(iv) are satisfied. Moreover, one has
\[
\partial_{x_i} \pip_i^\eps=\frac{\partial_{x_i} \Psi^{2\varepsilon}_{x_i} }{\sum_j \Psi^{2\varepsilon}_{x_j}}-\Psi^{2\varepsilon}_{x_i}\frac{\sum_j \partial_{x_i} \Psi^{2\varepsilon}_{x_j}}{(\sum_j \Psi^{2\varepsilon}_{x_j})^2}.
\]
from which (v) easily follows, using once again the bounded overlap property.
\end{proof}

\subsection{Regularity of the $p$-energy forms}

We denote by $C_c(X)$ the space of continuous and compactly supported functions. We recall that the metric space $(X,d)$ is said to be proper if closed metric balls are compact.

\begin{proposition}\label{p regularity}
    Assume that $(X,d)$ is proper. Let $p>1$. The $p$-energy form $\mathcal{E}_p$ with domain $W^{1,p}(X,m)$ is regular, that is, the space $C_c(X) \cap W^{1,p}(X,m)$ is dense in $C_c(X)$ for the supremum norm and dense in $W^{1,p}(X,m)$ for the $\| \cdot \|_{W^{1,p}(X,m)}$ norm.
\end{proposition}

\begin{proof}
We consider a partition of unity $\pip_i^\varepsilon$ as in Theorem \ref{partition unity}.  For $f \in C_c(X)$, consider
\[
\hat{f}_\varepsilon=\sum_i f(x_i)\varphi_i^\varepsilon.
\]
Since $f\in C_c(X)$, it is clear from the properness assumption that $\hat{f}_\varepsilon \in C_c(X)\cap W^{1,p}(X,m)$. Proving that $f_\varepsilon$ converges to $f$ for the supremum norm topology easily follows from uniform continuity arguments. Therefore, $C_c(X) \cap W^{1,p}(X,m)$ is dense in $C_c(X)$ for the supremum norm. Next, let $\eta>0$ and $f \in W^{1,p}(X,m)$. One can find a ball $B_R$ with radius $R$ large enough so that
\[
\int_{X \setminus B_R} |f|^p dm+\int_{\mathcal{S} \setminus B_R} |\partial f|^p d\nu <\eta
\]
Consider the function
\[
\Phi=\sum_{i: B_{R+4\varepsilon} \cap B(x_i,2\varepsilon )\neq \emptyset} \pip_i^{\varepsilon/2} \in C_c(X) \cap W^{1,p}(X,m),
\]
where $\varepsilon>0$ is fixed and small enough and $B_{R+4\varepsilon}$ is the ball with the same center as $B_R$ and radius $R+4\varepsilon$. Since $\Phi=1$ on $B_{R+4\varepsilon}$, we have $\nu$-a.e. $\partial \Phi=0$ outside of the union of the balls $B(x_i,2\varepsilon)$ which are included in $X \setminus B_R$. For those balls, it follows from Theorem \ref{thm:Morrey} that
\[
\sup_{B(x_i,2\varepsilon)} |f|^p \le 2^{p-1} \left( \frac{1}{m(B(x_i,2\varepsilon))} \int_{B(x_i,2\varepsilon)} |f|^p dm+(4\varepsilon)^{p-1} \int_{\mathcal{S} \cap B(x_i,2\varepsilon)} |\partial f|^p d\nu \right).
\]
Using then the bounded overlap property and (v) in Theorem \ref{partition unity}, it is easy to check that we have $f \Phi \in C_c(X) \cap W^{1,p}(X,m)$ and
\[
\| f\Phi -f \|^p_{L^p(X,m)}+\| \partial(f\Phi -f) \|^p_{L^p(\mathcal{S},\nu)} \le C \eta,
\]
where the constant $C$ depends on $\varepsilon$ but not on $\eta$.
\end{proof}

\subsection{Korevaar-Schoen characterization of the Sobolev and BV spaces}

We now turn to the characterization of the Sobolev and BV spaces as Korevaar-Schoen type spaces. For $p \ge 1, r>0,$ and $f \in L^p(X,m)$, we define 

\[
E_{p,\Psi_p} (f,r) = \frac{1}{\Psi_p(r)}\int_X \frac{1}{m(B(x,r))}\int_{B(x,r)} |f(y)-f(x)|^p dm(y) dm(x),
\]
where $\Psi_p(r)=r^{p-1}\Phi(r)$. We observe that
\[
\frac{1}{C}\left(\frac{R}{r}\right)^p\le\frac{\Psi_p(R)}{\Psi_p(r)}\le C\left(\frac{R}{r}\right)^{p-1}\frac{\Phi(R)}{\Phi(r)}\text{ for any }r,R\in (0,r_0]\text{ with }r\le R.
\]

Define the following Korevaar-Schoen  space
\begin{align}\label{definition KS}
KS^{\Psi_p,p}(X,m) =\left\{ f \in L^p(X,m), \limsup_{r \to 0} E_{p,\Psi_p} (f,r)<+\infty\right\}.
\end{align}


\begin{theorem}\label{weak monotonicity Lp}
Let $p>1$. Then $KS^{\Psi_p,p}(X,m) = W^{1,p}(X,m)$ and there exist  constants $c,C,r_0>0$ such that for every $f \in KS^{\Psi_p,p}(X,m)$
\[
c \sup_{r \in (0,r_0] }E_{p,\Psi_p} (f,r)  \le \int_\mathcal{S} | \partial f |^p d\nu \le C\liminf_{r  \to 0 }E_{p,\Psi_p} (f,r) 
\]
\end{theorem}

\begin{proof}
We adapt to our setting  arguments from \cite[Theorem 5.2]{Bau22}. Thanks to Lemma \ref{cover:local}, we pick constants $C$ and $0<r_0 < \frac{\underline{\iota}}{2}$ small enough so that the uniform volume control holds for $0<r<2r_0$ and such that for every $r \in (0,r_0)$ there exists an at most countable cover of $X$ made of balls $B_i$ of radius $r$ such that each point of $X$ belongs to at most $C$ of the balls $2B_i$. From Theorem \ref{thm:Morrey}, for $f \in W^{1,p}(X,m)$ and $r<r_0$ we have then:
\begin{align*}
E_{p,\Psi_p}(f,r)&=\frac{1}{\Psi_p(r)}\int_X \frac{1}{m(B(x,r))}\int_{B(x,r)} |f(y)-f(x)|^p dm(y) dm(x)\\
& \le \frac{C}{r^{p-1}\Phi(r)^2} \int_X  \int_{B(x,r)} |f(y)-f(x)|^p dm(y) dm(x)\\
 & \le \sum_i \frac{C}{r^{p-1}\Phi(r)^2} \int_{B_i}  \int_{B(x,r)} |f(y)-f(x)|^p dm(y) dm(x)
 \\
 & \le \sum_i \frac{C}{r^{p -1}\Phi(r)^2} \int_{B_i} \int_{B(x,r)}  \int_{]x,y[} d(x,y)^{p-1} |\partial f|^p(z) d\nu(z) dm(y) dm(x)\\
 & \le \sum_i \frac{C}{\Phi(r)^2} \int_{B_i} \int_{B(x,r)}  \int_{]x,y[}  |\partial f|^p(z) d\nu(z) dm(y) dm(x)
 \\
 & \le \sum_i \frac{C}{\Phi(r)^2} \int_{B_i} \int_{2B_i}  \int_{2B_i}  |\partial f|^p(z) d\nu(z) dm(y) dm(x)\\
 & \le C \sum_i  \int_{2B_i}  |\partial f|^p(z) d\nu(z)\\
 & \le C   \int_{\mathcal{S}}  |\partial f|^p(z) d\nu(z)
\end{align*}

Therefore $W^{1,p}(X,m) \subset KS^{\Psi_p,p}(X,m) $, and it remains to prove the converse inclusion. We consider a partition of unity $\pip_i^\varepsilon$ as in Theorem \ref{partition unity} and for $f \in KS^{\Psi_p,p}(X,m)$ we define
\begin{align}\label{discrete convolution}
f_\varepsilon=\sum_i  f_{B_i^\eps} \pip_i^\varepsilon
\end{align}
where $f_{B_i^\eps}=\frac{1}{m(B_i^\eps)}\int_{B_i^\eps} f dm$. Due to the bounded overlap property, for any $x\in B_j^\eps$ it holds that
\begin{align*}
    |f_{\eps}(x)-f(x)|&=\Big| \sum_{i:B_i^{4\eps }\cap B_j^{4\eps }\ne\emptyset}\varphi_i^\eps (x)(f_{B_i^\eps}-f(x))\Big|\\
    &\leq \sum_{i:B_i^{4\eps}\cap B_j^{4\eps }\ne\emptyset}\Big|\frac{1}{m(B_i^\eps)}\int_{B_i^\eps}(f(y)-f(x))\,dm(y)\Big|\\
    &\leq \sum_{i:B_i^{4\eps }\cap B_j^{4\eps }\ne\emptyset}\frac{1}{m(B_i^\eps)}\int_{B_i^\eps}|f(y)-f(x)|\,dm(y)\\
    &\leq \frac{C}{m(B(x,12\eps ))}\int_{B(x,12\eps )}|f(x)-f(y)|dm(y),
\end{align*}
whence
\begin{align}
\|f_{\eps}-f\|_{L^p(X,m)}^p & 
    = \int_X |f_\varepsilon(x)-f(x)|^p dm(x)\\
  &\le C  \int_X\left( \frac{1}{m(B(x,12\eps ))}\int_{B(x,12\eps )}|f(x)-f(y)|dm(y) \right)^pdm(x) \\
  &\le C  \int_X \frac{1}{m(B(x,12\eps ))}\int_{B(x,12\eps )}|f(x)-f(y)|^p dm(y) dm(x) \\
  &\le C \Psi_p(12\varepsilon )  E_{p,\Psi_p}(f,12\varepsilon ) \label{estimate :C4}
\end{align}

Therefore, from the definition of $KS^{\Psi_p,p}(X,m)$,  $f_\varepsilon$ converges to $f$ in $L^p(X,m)$ when $\varepsilon \to 0$. Next, for  $x\in B_j^\eps$, we have that
\begin{align*}
    f_{\eps}(x)-f_{B_j^\eps}= \sum_{i:B_i^{4\eps }\cap B_j^{4\eps }\ne\emptyset}\varphi_i^\eps (x)(f_{B_i^\eps}-f_{B_j^\eps})
    \end{align*}
    It  follows that $f_\varepsilon \in AC(X)$ and that    we have for $\nu$-a.e. $x\in B_j^\eps \cap \mathcal{S}$
    \[
    | \partial f_\varepsilon| (x)\le \sum_{i:B_i^{4\eps }\cap B_j^{4\eps }\ne\emptyset}| \partial \varphi_i^\eps (x)|\, |f_{B_i^\eps}-f_{B_j^\eps}|
    \]
Therefore, from the bounded overlap property and the basic convexity inequality
\[
\left(\sum_{i=1}^n x_i\right)^p \le n^{p-1} \sum_{i=1} x_i^p
\]
we have for $\nu$-a.e. $x\in B_j^\eps \cap \mathcal{S}$
\begin{align*}
| \partial f_\varepsilon| (x)^p &\le C \sum_{i:B_i^{4\eps }\cap B_j^{4\eps }\ne\emptyset}| \partial \varphi_i^\eps (x)|^p\, |f_{B_i^\eps}-f_{B_j^\eps}|^p \\
 &\le C \sum_{i:B_i^{4\eps }\cap B_j^{4\eps }\ne\emptyset}| \partial \varphi_i^\eps (x)|^p\, \left| \frac{1}{m (B_i^\eps) m(B_j^\eps)}\int_{B_i^\eps}\int_{B_j^\eps} (f(y) -f(z)) dm(y)dm(z) \right|^p \\
 &\le C \sum_{i:B_i^{4\eps }\cap B_j^{4\eps }\ne\emptyset}| \partial \varphi_i^\eps (x)|^p\,  \frac{1}{m (B_i^\eps) m(B_j^\eps)}\int_{B_i^\eps}\int_{B_j^\eps} |f(y) -f(z)|^p dm(y)dm(z)   \\
 &\le \frac{C}{\Phi(\varepsilon)} \sum_{i:B_i^{4\eps }\cap B_j^{4\eps }\ne\emptyset}| \partial \varphi_i^\eps (x)|^p\,  \int_{B_j^\eps} \frac{1}{m (B(z,12\varepsilon ))}\int_{B(z,12 \varepsilon )} |f(y) -f(z)|^p dm(y)dm(z),
\end{align*}

where we used the uniform volume growth estimate in the last inequality. As a consequence, and since
\[
 \int_\mathcal{S} |\partial \pip_i^\eps |^p d\nu  \le  \frac{C}{\eps^{ p-1}}
\]
we deduce
\begin{align*}
\int_{B_j^\eps \cap \mathcal{S}} | \partial f_\varepsilon| (x)^p d\nu(x) & \le \frac{C}{\Psi_p(\varepsilon)} \int_{B_j^\eps} \frac{1}{m (B(z,12\varepsilon ))}\int_{B(z,12\varepsilon )} |f(y) -f(z)|^p dm(y)dm(z). 
\end{align*}
It follows then that
\begin{align*}
\int_{ \mathcal{S}} | \partial f_\varepsilon| (x)^p d\nu(x) &\le \sum_j \int_{B_j^\eps \cap \mathcal{S}} | \partial f_\varepsilon| (x)^p d\nu(x) \\
&\le \frac{C}{\Psi_p(\varepsilon)} \int_{X} \frac{1}{m (B(z,12\varepsilon ))}\int_{B(z,12\varepsilon )} |f(y) -f(z)|^p dm(y)dm(z) \\
& \le C E_{p,\Psi_p}(f, 12 \varepsilon  ).
\end{align*}

Since $f_\varepsilon$ converges to $f$ in $L^p(X,m)$ when $\varepsilon \to 0$, by lower-semicontinuity of $\mathcal{E}_p$ (Corollary \ref{lower semi-continuity p-energy}), we conclude that $f \in W^{1,p}(X,m)$ and 
\[
\int_{ \mathcal{S}} | \partial f| (x)^p d\nu(x) \le \liminf_{\varepsilon \to 0} \int_{ \mathcal{S}} | \partial f_\varepsilon| (x)^p d\nu(x) \le C \liminf_{\varepsilon \to 0} E_{p,\Psi_p}(f, \varepsilon ).
\]

\end{proof}

For $p=1$, we obtain a similar result.

\begin{theorem}\label{characterization BV}
\ $KS^{\Psi_1,1}(X,m) = BV(X,m)$ and there exist  constants $c,C>0$ such that for every $f \in KS^{\Psi_1,1}(X)$
\[
c \sup_{r \in (0,r_0] }E_{1,\Psi_1} (f,r)  \le \| \partial f \|(X) \le C\liminf_{r  \to 0 }E_{1,\Psi_1} (f,r) .
\]

\end{theorem}

\begin{proof}
The proof is very close to  that of Theorem \ref{weak monotonicity Lp}.
\begin{enumerate}
\item Arguing as before, we get that for every $f \in W^{1,1}(X,m)$
\begin{align*}
E_{1,\Psi_1}(f,r) \le C   \int_{\mathcal{S}}  |\partial f|(z) d\nu(z)
\end{align*}
Let now $f \in L^1(X,m)$ and let $f_n \in W^{1,1}(X,m)$ that converges to $f$ in $L^1(X,m)$. We have
\begin{align*}
E_{1,\Psi_1}(f_n,r) \le C   \int_{\mathcal{S}}  |\partial f_n|(z) d\nu(z)
\end{align*}
so that when $n \to +\infty$,
\[
E_{1,\Psi_1}(f,r) \le C   \liminf_{n \to +\infty} \int_{\mathcal{S}}  |\partial f_n|(z) d\nu(z).
\]
We conclude
\[
E_{1,\Psi_1}(f,r) \le C \| \partial f \|(X).
\]
\item With the same notations as in the proof of Theorem \ref{weak monotonicity Lp} one has
\begin{align*}
\int_{ \mathcal{S}} | \partial f_\varepsilon| (x) d\nu(x) \le  C E_{1,\Psi_1}(f, 12\varepsilon  ).
\end{align*}
Therefore, one has
\[
 \| \partial f \|(X)\le C \liminf_{\varepsilon \to 0}  E_{1,\Psi_1}(f, \varepsilon ).
\]
\end{enumerate}
\end{proof}

\begin{corollary}\label{Banach space property}
 $(BV(X,m), \| \cdot \|_{BV(X)})$ is a Banach space.
\end{corollary}

\begin{proof}
From Theorem \ref{characterization BV},  one has $BV(X,m)=KS^{\Psi_1,1}(X,m)$ and the norm $\| \cdot \|_{BV(X)}$ is equivalent to the norm
\[
\| f \|_{KS^{\Psi_1,1}(X,m)}:=\| f \|_{L^1(X,m)} + \sup_{r \in (0,r_0] }E_{1,\Psi_1} (f,r).
\]
Since $(KS^{\Psi_1,1}(X,m), \| \cdot \|_{KS^{\Psi_1,1}(X,m)})$ is a Banach space, see the proof of \cite[Theorem 3.1]{Bau22}, the conclusion follows.
\end{proof}

\subsection{Critical exponents and real interpolation of the Besov-Lipschitz spaces}

Given $p \ge 1$ and $\alpha \ge 0$, we consider the Besov-Lipschitz space
\[
\mathcal{B}^{\alpha,p}(X,m) =\left\{ f \in L^p(X,m) \mid  \sup_{0<r <r_0 } E_{p,\alpha} (f,r)<+\infty\right\}
\]
where $r_0>0$ is as in Theorem \ref{weak monotonicity Lp} and $E_{p,\alpha}(f,r)$ is defined by 

\begin{align}\label{KS functional}
E_{p,\alpha} (f,r) = \int_X \frac{1}{m(B(x,r))}\left( \int_{B(x,r)} \frac{|f(y)-f(x)|^p}{r^{p\alpha}} dm(y) \right)dm(x) 
\end{align} 

Several properties of the Besov-Lipschitz spaces are pointed out  in \cite[Section 3]{Bau22} in the context of general metric spaces. In this section, as an application, of the previous results we show that on local trees they form a real interpolation scale of spaces. We first point out the following lemma about the critical exponents of the family of   Besov-Lipschitz spaces.

\begin{lemma}
Assume  the volume growth function $\Phi$ is given by
\[
\Phi (r)=r^{d_h}
\]
with $d_h \ge 1$.
 Let $p \ge 1$. Then,
\[
\sup \left\{ \alpha \ge 0  \mid \mathcal{B}^{\alpha,p}(X,m) \text{ contains non-constant fuctions }  \right\}=1+\frac{d_h-1}{p}.
\]
\end{lemma}

\begin{proof}
Indeed, let $\alpha_p =1+\frac{d_h-1}{p}$. If $\alpha =\alpha_p$, then $\mathcal{B}^{\alpha,p}(X,m)$ is the Korevaar-Schoen space $KS^{\Psi_p,p}(X,m)$, which contains non-constant functions, being dense in $L^p(X,m)$. On the other hand let $\alpha > \alpha_p$ and $f \in \mathcal{B}^{\alpha,p}(X,m)$. One has for $0<r<r_0$ 
\[
 \int_X \frac{1}{m(B(x,r))}\left( \int_{B(x,r)} \frac{|f(y)-f(x)|^p}{r^{p\alpha}} dm(y) \right)dm(x)  \le C.
\]
Therefore $\liminf_{r \to 0} E_{p,\alpha_p} (f,r)=0$. From Theorem \ref{weak monotonicity Lp} if $p>1$ and Theorem \ref{characterization BV} if $p=1$, this implies $$\sup_{r \in (0,r_0]} E_{p,\alpha_p} (f,r)=0,$$ so that $f$ is constant.
\end{proof}

We now recall some basic definitions of the real interpolation theory for seminormed spaces as in  \cite{GKS}. Let $X_0$ and $X_1$ be two seminormed spaces which are subspaces of a same vector space. The sum $X_0+X_1$ is equipped with the seminorm 
\[
\|f\|_{X_0+X_1}:=\inf\{\|f_0\|_{X_0}+\|f_1\|_{X_1}, f=f_0+f_1\}.
\]

The $K$-functional of $(X_0,X_1)$ is defined for each $f\in X_0+X_1$ and $t>0$ by 
\[
K(f,t,X_0,X_1):=\inf\{\|f_0\|_{X_0}+t\|f_1\|_{X_1} \mid  f=f_0+f_1\}.
\]

Suppose that $0<\theta<1$, $1\le q<\infty$ or $0\le \theta\le 1$, $ q=\infty$. Then the interpolation space $(X_0,X_1)_{\theta,q}$ consists of functions $f\in X_0+X_1$ such that 
\[
\|f\|_{\theta,q}=
\begin{cases}
\left(\int_0^{\infty}(t^{-\theta}K(f,t,X_0,X_1))^q\frac{dt}{t}\right)^{1/q}, &0<\theta<1,1\le q<\infty, \\
\sup_{t>0} t^{-\theta}K(f,t,X_0,X_1), &0\le \theta\le 1, q=\infty,
\end{cases}
\]
is finite. In our setting, we take $X_0=L^p(X,m)$ and $X_1=W^{1,p}(X,m)$ is equipped with the  seminorm $\| \partial f \|_{L^p(\mathcal{S},\nu)}$ where $p > 1$ (for $p=1$, we take $X_1=BV(X,m)$ equipped with the  seminorm $\| \partial f \|(X)$). For $t \ge 0$, we simply denote
\[
K(f,t)=\inf\{\| g \|_{L^p(X,m)}+t\|\partial h\|_{L^p(\mathcal{S},\nu)} \mid  f=g+h\},
\]
with $\|\partial h\|_{L^1(\mathcal{S},\nu)}$ replaced by $\| \partial h \|(X)$ if $p=1$. If $f=g+h$ with $g \in L^p(X,m)$ and $h \in W^{1,p}(X,m)$, then $\| h \|_{L^p(X,m)} \le \| f \|_{L^p(X,m)}+\| g \|_{L^p(X,m)}$. Therefore, for $f \in L^p(X,m)$ and $0<t \le 1$,
\[
K(f,t) \le K(f,t,L^p(X,m),W^{1,p}(X,m)) \le 2K(f,t)+t \| f \|_{L^p(X,m)},
\]
where, in the middle term, $W^{1,p}(X,m)$ is equipped with the norm $\| \cdot \|_{W^{1,p}(X,m)}$ (and is replaced by $BV(X,m)$ equipped with the norm $\| \cdot \|_{BV(X)}$ if $p=1$). Since both $K$-functionals are bounded by $\| f \|_{L^p(X,m)}$, the interpolation spaces $(X_0,X_1)_{\theta,q}$ are the same whether $X_1$ is equipped with its seminorm or with its norm.

We have then the following theorem.

\begin{theorem}
Assume  the volume growth function $\Phi$ is given by
\[
\Phi (r)=r^{d_h}
\]
with $d_h \ge 1$. Let $r_0>0$ be as in Theorem  \ref{weak monotonicity Lp}.  There exist  $C_1, C_2>0$   such that for any $f\in L^p(X,m)+ W^{1,p}(X,m)$ if $p>1$ and $f\in L^1(X,m)+ BV(X,m)$ if $p=1$ the following holds for $0<r<r_0$,
\[
C_1 E_{p,0}(f,r)^{\frac1p} \le K(f,r^{\alpha_p}) \le C_2 E_{p,0}(f,r)^{\frac1p},
\]
where $\alpha_p=1+\frac{d_h-1}{p} $. Therefore, for   $0\le \alpha \le\alpha_p$
\begin{align*}
\mathcal{B}^{\alpha,p}(X,m)=
\begin{cases}
(L^1(X,m), BV(X,m))_{\alpha/\alpha_1,\infty}, &  p=1 \\
(L^p(X,m), W^{1,p}(X,m))_{\alpha/\alpha_p,\infty}, &  p>1.
\end{cases}
\end{align*}
\end{theorem}

\begin{proof}
We assume $p>1$. The case $p=1$ follows in the same way. We first prove the bound
\[
  E_{p,0}(f,r)^{\frac1p}  \le C K(f,r^{\alpha_p}).
\]
Let $f=g+h$, where $g\in L^p(X,m)$ and $h\in W^{1,p}(X,m)$.
Then by Minkowski's inequality and Theorem \ref{weak monotonicity Lp} we have for $0<r<r_0$
\begin{align*}
     E_{p,0}(f,r)^{\frac1p}  
     & \le E_{p,0}(g,r)^{\frac1p}+ E_{p,0}(h,r)^{\frac1p} \\
     &\le  C\left(\|g\|_{L^p(X,m)}+r^{\alpha_p}r^{-\alpha_p}E_{p,0}(h,r)^{\frac1p}\right) \\
      &\le  C\left(\|g\|_{L^p(X,m)}+r^{\alpha_p}\|\partial h\|_{L^p(\mathcal{S},\nu)}\right).
\end{align*}

We now turn to the proof of the second inequality, that is, $K(f,r^{\alpha_p}) \le C  E_{p,0}(f,r)^{\frac1p}$. We consider a partition of unity $\pip_i^\varepsilon$ as in Theorem \ref{partition unity} and for $f \in L^p(X,m)$ we define as before
\begin{align}\label{discrete convolution2}
f_\varepsilon=\sum_i  f_{B_i^\eps} \pip_i^\varepsilon
\end{align}
where $f_{B_i^\eps}=\frac{1}{m(B_i^\eps)}\int_{B_i^\eps} f dm$. Set $g_\varepsilon=f-f_\varepsilon$ so that $f=g_\varepsilon+f_\varepsilon$.  We claim that $g_\varepsilon \in L^p(X,m)$ and $f_\varepsilon \in W^{1,p}(X,m)$ and moreover that both $\|g_\varepsilon \|_{L^p(X,m)}$ and $\| \partial f_\varepsilon \|_{L^p(\mathcal{S},\nu)}$ can be bounded in terms of $E_{p,0}(f,r)$. 
Indeed, from the proof of Theorem \ref{weak monotonicity Lp} we have
\begin{align*}
\| g_\varepsilon\|_{L^p(X,m)}^p&= \|f_{\eps}-f\|_{L^p(X,m)}^p  \\
 &\le C  \int_X \frac{1}{m(B(x,12\eps ))}\int_{B(x,12\eps )}|f(x)-f(y)|^p dm(y) dm(x) \\
  &\le C E_{p,0}(f,12\varepsilon ).
\end{align*}

Again from the proof of Theorem \ref{weak monotonicity Lp} we have
\begin{align*}
\int_{ \mathcal{S}} | \partial f_\varepsilon| (x)^p d\nu(x) &\le \frac{C}{\varepsilon^{d_h+p-1}} \int_{X} \frac{1}{m (B(z,12\varepsilon ))}\int_{B(z,12\varepsilon )} |f(y) -f(z)|^p dm(y)dm(z) \\
& \le \frac{C}{\varepsilon^{d_h+p-1}}  E_{p,0}(f, 12\varepsilon  ).
\end{align*}

We conclude that for every $r=12\varepsilon<r_0 $ 
\[
K(f,r^{\alpha_p}) \le \| g_\varepsilon \|_{L^p(X,m)}+r^{\alpha_p}\|\partial f_\varepsilon\|_{L^p(\mathcal{S},\nu)}\le C E_{p,0}(f,r)^{\frac1p}.
\]
This establishes the first part of the statement. For the second part of the statement, for $f \in L^p(X,m)$ we can write $f=g+h$ with $g=f$ and $h=0$ and deduce that for every  $t > 0$
\[
K(f,t) \le \| f \|_{L^p(X,m)}.
\]
The conclusion follows by definition of the interpolation space and of its seminorm.
\end{proof}

\subsection{Nash inequality}

\begin{theorem}\label{Nash local tree} 
Assume the volume growth function $\Phi$ is given by
\[
\Phi (r)=r^{d_h}
\]
with $d_h \ge 1$. Let $p>1$. The following Nash inequality holds for every $f \in W^{1,p}(X,m)$,
\[
\| f \|_{L^p(X,m)} \le C \left(\| f \|_{L^p(X,m)}+\| \partial f \|_{L^p(\mathcal S,\nu)}\right)^{\theta} \| f \|^{1-\theta}_{L^1(X,m)}
\]
where $\theta=\frac{(p-1)d_h}{p-1+pd_h}$. 
\end{theorem}

\begin{proof}
We first note the following two easily proved properties which indicate that the Sobolev norm behaves nicely under cutoffs. In what follows we denote $a\wedge b=\min (a,b)$ and $a_+=\max (a,0)$.
\begin{itemize}
\item For every $s,t \ge 0$, and $f\in W^{1,p}(X,m)$,  
\[
\int_\mathcal{S} |\partial((f-t)_+ \wedge s )|^p d\nu \le \int_\mathcal{S} | \partial f|^p d\nu.
\]
\item  For any non-negative $f \in W^{1,p}(X,m)$  and any $\rho>1$,
\[
 \sum_{k \in \mathbb{Z}} \int_\mathcal{S} |\partial f_{\rho,k}|^p d\nu \le  \int_\mathcal{S} |\partial f|^p d\nu,
\]
where $f_{\rho,k}:=(f-\rho^k)_+ \wedge \rho^k(\rho-1)$, $k \in \mathbb{Z}$. 
\end{itemize}
On the other hand, consider the averaging operator
\[
\mathcal{M}_rf(x)=\frac{1}{m(B(x,r))}\int_{B(x,r)} f dm
\]
It immediately follows from the uniform volume control that for any $f \in L^1(X,m)$ one has for $0<r<r_0$
\[
\| \mathcal{M}_r f \|_{L^\infty(X,m)} \le \frac{C}{r^{d_h}} \| f \|_{L^1(X,m)}.
\]
From H\"older's inequality and Theorem \ref{weak monotonicity Lp} we also have for every $f\in W^{1,p}(X,m)$ and  $0<r<r_0$
\[
\| f -\mathcal{M}_r f \|_{L^p(X,m)} \le  r^{\alpha_p} \, \sup_{\rho \in (0,r_0)} E_{p,\alpha_p} (f,\rho)^{1/p} \le C r^{\alpha_p} \, \left(\int_{\mathcal{S}} |\partial f|^p d\nu\right)^{1/p}.
\]
The previous observations are enough to obtain the full scale of Gagliardo-Nirenberg inequalities and in particular the stated Nash inequalities:  This follows from applying the results of \cite[Theorem 9.1]{MR1386760} . 
\end{proof}

\section{Heat kernel characterization of the  Sobolev and BV spaces}

Everywhere in this section, we assume that the metric space $(X,d)$ is proper and that the assumptions  of Section \ref{KS section} are satisfied. In particular, one has for every $x \in X$ and $0<r<r_0$

$$\frac{1}{C}\Phi(r)\le m(B(x,r))\le C\Phi(r).$$

As before, we denote $\Psi_p(r)=r^{p-1}\Phi(r)$ so that $\Psi_2(r)=r\Phi(r)$. Up to replacing $\Phi(r)$ by the function $\int_0^r \Phi(s) \frac{ds}{s}$, which is comparable to $\Phi$ thanks to \eqref{volume assumption}, we can and will assume that $\Phi$ is continuous and strictly increasing with $\Phi(0)=0$. We moreover extend $\Phi$ to $[0,+\infty)$ by setting $\Phi(r)=\Phi(r_0)r/r_0$ for $r > r_0$, so that \eqref{volume assumption} holds for any $0<r\le R$ (with possibly different constants) and the functions $\Phi^{-1}$ and $\Psi_2^{-1}$ are well defined on $[0,+\infty)$.  Our goal in this section is to construct a heat kernel $p_t(x,y)$ on $(X,d,m)$ and prove that it satisfies the following estimates: 

\begin{itemize}
\item \textbf{Upper bound:} For $t \in (0,t_0], d(x,y) \le \underline{\iota}/2$, 
\begin{align}
p_t(x,y) \le \frac{C_1}{\Phi (\Psi_2^{-1}(t))} \exp \left(-C_2 \frac{d(x,y)}{\Phi^{-1} \left( t/d(x,y) \right)} \right).  \label{UB}
\end{align}
\item \textbf{Near diagonal lower bound:} For $\Psi_2(c_2d(x,y) )\le t \le t_0$,
\begin{align}
p_t(x,y) \ge \frac{c_1}{\Phi (\Psi_2^{-1}(t))}. \label{NDLB}
\end{align}
\item \textbf{Escape rate estimate:} For $t \in (0,t_0], r \le \underline{\iota}/2$,
\begin{align}\label{ERE}
\int_{X \setminus B(x,r)} p_t(x,y) dm(y) \le C_1 \exp \left(-C_2 \frac{r}{\Phi^{-1} \left( t/r \right)} \right). 
\end{align}
\end{itemize}

Those estimates will be obtained by  localizing some of the arguments in \cite{KumagaiPRIMS}. We point out that some of the arguments have also been a little simplifies like the proof of Theorem \ref{thm:upper-Holder}. As a consequence of those estimates, in the second part of the section we will then obtain a heat kernel characterization of the class $BV(X,m)$ and  the Sobolev spaces $W^{1,p}(X,m)$, $p >1$.

\subsection{Heat kernel estimates}\label{Heat kernel estimates}

\begin{lemma}
The 2-energy form
\[
\mathcal{E}_2 (f)=\int_{\mathcal S} | \partial f |^2 d\nu
\]
with domain $W^{1,2}(X,m)$ is a regular and strongly local Dirichlet form in $L^2(X,m)$. 
\end{lemma}

\begin{proof}
We first note that for $f, g \in W^{1,2}(X,m)$ 
\[
\frac{1}{4} ( \mathcal{E}_2 (f+g)- \mathcal{E}_2 (f-g))=\int_{\mathcal S} (\partial f ) (\partial g )  d\nu
\]
is indeed a bilinear form. Then, $W^{1,2}(X,m)$ is dense in $L^2(X,m)$ and  $\mathcal{E}_2$ is closed in $L^2(X,m)$ from Proposition \ref{closedness p-energy}. Finally, if $u \in W^{1,2}(X,m)$ and $v=\min ( \max (u,0),1)$, then for $x,y \in X$
\[
|v(x) -v(y)| \le |u(x) -u(y)|,
\]
 and it easily follows that $v \in W^{1,2}(X,m)$ and that $\nu$-a.e. $| \partial v| \le |\partial u|$ so that $\mathcal{E}_2 (v) \le \mathcal{E}_2 (u)$. Therefore, we conclude that $\mathcal{E}_2 $ is a Dirichlet form and its regularity follows from Proposition \ref{p regularity}. It is strongly local because if $f \in W^{1,2}(X,m)$ is constant on an open set $U$, then $\nu$-a.e. $\partial f=0$ on $U \cap \mathcal{S}$.
\end{proof}

\begin{proposition}
The Dirichlet form $\mathcal{E}_2$ admits a heat kernel $p_t(x,y)$  which satisfies for every $x,y,z \in X$, $d(y,z) < \underline{\iota}$, $t>0$,
\begin{equation}\label{eq:Holder}
|p_t(x,y) -p_t(x,z)|^2\le \frac{d(y,z)}{t} p_{t}(x,x).  
\end{equation}
\end{proposition}

\begin{proof}
The existence of the heat kernel follows from the fact that the Dirichlet form $\mathcal{E}_2$ has a domain included in the space of continuous functions and the estimate \eqref{eq:Holder} follows from Theorem \ref{thm:Morrey} applied to $f=p_t(x,\cdot)$ and the estimate $\mathcal{E}_2( p_t(x,\cdot)) \le \frac{1}{t} \int_X p_{t/2}(x,z)^2 dm(z)=\frac{1}{t} p_t(x,x)$ which comes from spectral theory. 
\end{proof}


\begin{theorem}\label{thm:upper-Holder}
The Dirichlet form $\mathcal{E}_2$ admits a heat kernel $p_t(x,y)$  which satisfies for every $t>0$, $x,y \in X$
\begin{equation}\label{eq:on-diag}
p_t(x,y) \le \frac{C}{\Phi (\Psi_2^{-1}(t))\wedge 1}.
\end{equation}

\end{theorem}
\begin{proof}
    Let $x\in X$. For any ball $B(x,r)$ in $X$, one has $\int_{B(x,r)}p_t(x,y)dm(y)\le 1$. Hence there exists $y=y(t,r)\in B(x,r)$ such that $p_t(x,y)\le m(B(x,r))^{-1}$. In particular, the volume growth assumption  implies that for any $0<r<\min (r_0,\underline{\iota})$
    \[
    p_t(x,y)\le \frac{C}{\Phi(r)}.
    \]
    Now using this estimate, together with the H\"older regularity \eqref{eq:Holder}, we obtain 
    \[
    \frac12 p_t(x,x)^2\le p_t(x,y)^2+|p_t(x,x)-p_t(x,y)|^2\le \frac{C^2}{\Phi(r)^2}+\frac{r}{t}p_t(x,x).
    \]
    The above quadratic equation gives that 
    \begin{equation}\label{eq:upper-OD1}
    p_t(x,x)\le \frac{r}{t}+\sqrt{\frac{r^2}{t^2}+\frac{C^2}{\Phi(r)^2}}.
    \end{equation}
    Now letting $t=\Psi_2(r)=r\Phi(r)$, we have that 
    \(
    p_{r\Phi(r)}(x,x)\le \frac{C_1}{\Phi(r)}
    \)
    and hence for any $t< \Psi_2(\min (r_0,\underline{\iota}))$ there holds
    \[
    p_t(x,x)\le \frac{C_1}{\Phi(\Psi_2^{-1}(t))}.
    \]
    On the other hand, the upper bound \eqref{eq:upper-OD1} holds for any $t>0$ by taking $r=\min (r_0,\underline{\iota})/2$, i.e., $p_t(x,x)\le C_{t}$ for a decreasing function $C_t$. We thus conclude that 
    \[
    p_t(x,x)\le \frac{C}{\Phi (\Psi_2^{-1}(t))\wedge 1}
    \]
    and by the semigroup property
    \[
    p_t(x,y)\le p_t(x,x)^{1/2}p_t(y,y)^{1/2}\le \frac{C}{\Phi (\Psi_2^{-1}(t))\wedge 1}, \quad \forall x,y\in X \text{ and } t>0.
    \]
\end{proof}

We denote by $(X_t)_{t \ge 0}$ the diffusion process associated with the regular and strongly local Dirichlet form $\mathcal{E}_2$, see \cite{FOT}. Note that from Theorem \ref{thm:Morrey} and the uniform volume control, points have a positive capacity, so that the process can be started from any point $x \in X$; we denote by $\mathbb P^x$ its law and by $\mathbb E^x$ the corresponding expectation. The proof for the following lemma is the same as in \cite[Section 4]{KumagaiPRIMS}. 
\begin{lemma}\label{lem:exit}
  There exists $0<r_1\le \min (r_0,\underline{\iota})$ such that for any $x\in X$ and $0<r\le r_1$, it holds that 
  \[
  \mathbb E^x[T_{B(x,r)}]\simeq \Psi_2(r),
  \]
  where $T_A$ ($A\subset X$) denotes the first exit time from $A$, i.e., $T_A=\inf\{t\ge 0: X_t\notin A\}$.
\end{lemma}

\begin{proof}
For the sake of completeness, we explain how the resistance estimates used in \cite[Section 4]{KumagaiPRIMS} are obtained in our setting. Let $r_1>0$ be small enough, $x \in X$, $0<r\le r_1$ and $B=B(x,r)$. Let $g_B$ be the Green function of the Dirichlet form $\mathcal{E}_2$ killed outside of $B$, so that $\mathbb E^y[T_B]=\int_B g_B(y,w) dm(w)$ for $y \in B$ and
\[
g_B(y,y)=\sup \left\{ \frac{u(y)^2}{\mathcal{E}_2(u)} \mid u \in W^{1,2}(X,m), \, u=0 \text{ on } X \setminus B \right\}.
\]
Since $X$ is connected, there exists $z \in X$ such that $d(x,z)=r$. If $u \in W^{1,2}(X,m)$ vanishes on $X \setminus B$, then Theorem \ref{thm:Morrey} yields for $y \in B$, $u(y)^2=|u(y)-u(z)|^2 \le 2r \mathcal{E}_2(u)$. Therefore $g_B(y,y) \le 2r$ and thus $g_B(y,w) \le g_B(y,y)^{1/2}g_B(w,w)^{1/2} \le 2r$ for $y,w \in B$. On the other hand, the function $\Psi_x^{r/2}$ of Lemma \ref{bump function} vanishes on $X \setminus B$ and satisfies $\Psi_x^{r/2}(x)=1$ and $\mathcal{E}_2(\Psi_x^{r/2}) \le C/r$, so that $g_B(x,x) \ge r/C$. Since $\mathcal{E}_2(g_B(x,\cdot))=g_B(x,x)$, Theorem \ref{thm:Morrey} also yields $|g_B(x,x)-g_B(x,w)|^2 \le d(x,w) g_B(x,x)$ and therefore $g_B(x,w) \ge r/(2C)$ if $d(x,w) \le r/(4C)$. Integrating those bounds on $g_B$ and using the uniform volume control, we conclude
\[
\frac{1}{C} \Psi_2(r) \le \mathbb E^x[T_B] \le \sup_{y \in B} \mathbb E^y[T_B] \le C \Psi_2(r).
\]
\end{proof}

\begin{theorem}\label{thm:exit and UB}
For $t>0$, $x \in X$ and $r < \underline{\iota}$,
\begin{equation}\label{eq:exit-prob}
\mathbb P^x (T_{B(x,r)}\le t) \le C_1\exp\left(-C_2\frac{r}{\Phi^{-1}(t/r)} \right).
\end{equation}
Therefore, there exist two constants $C_1, C_2>0$ such that for any $0<t \le 1$ and $d(x,y)\le \underline{\iota}/2$, 
\[
p_t(x,y) \le \frac{C_1}{\Phi (\Psi_2^{-1}(t))} \exp \left(-C_2 \frac{d(x,y)}{\Phi^{-1} \left( t/d(x,y) \right)} \right).
\]    
\end{theorem}
\begin{proof}
When $r<\underline{\iota}$, the ball $B(x,r)$ is a tree. Thus, thanks to Lemma \ref{lem:exit}, \eqref{eq:exit-prob} follows as in the proof of \cite[Lemma 4.2]{KumagaiPRIMS}, see Example 1 in \cite[Section 5]{KumagaiPRIMS}. Note that \eqref{eq:exit-prob} implies the escape rate estimate \eqref{ERE} since $\int_{X \setminus B(x,r)} p_t(x,y) dm(y)\le \mathbb P^x(X_t\notin B(x,r))\le \mathbb P^x (T_{B(x,r)}\le t)$.

Now in view of \eqref{eq:exit-prob}, together with the uniform volume control and \eqref{eq:on-diag}, the  argument in \cite[Theorem 3.11]{BarlowLectures} yields that for any $0<t\le 1$ and $d(x,y)\le \underline{\iota}$
\[
p_t(x,y) \le \frac{C_1}{\Phi (\Psi_2^{-1}(t))} \exp \left(-C_2 \frac{d(x,y)}{\Phi^{-1} \left( t/d(x,y) \right)} \right).
\]
We skip here the details for the sake of conciseness.
\end{proof}
It remains to show the near diagonal  lower estimate \eqref{NDLB} for the heat kernel. We begin with the following on-diagonal lower bound. 
\begin{lemma}\label{lem:on-diagLB}
There exist two constant $c,C>0$ and $t_0>0$ (depending on $\underline{\iota}, r_0$) such that for $0<t<t_0$ 
    \begin{equation}\label{eq:on-diagLB}
    p_t(x,x)\ge \frac{C}{\Phi(\Psi_2^{-1}(t))}.
    \end{equation}
\end{lemma}
\begin{proof}
    We first note from Lemma \ref{lem:exit} that for any $r\le r_1/2$, there exists $0<c_1<1$ and $c_2>0$ such that 
    \[
    \mathbb P^x(X_t\notin B(x,r))\le \mathbb P^x(T_{B(x,r)}\le t)\le c_1+c_2t/\Psi_2(r).
    \]
    Now choose $r$ such that $\Psi_2(r)\simeq t$ and $c_1+c_2 t/\Psi_2(r)=c_3<1$. Indeed, this is possible as long as  $t<c_4\Psi_2(r_1/2)$ for a small constant $c_4>0$. Then we have 
    \[
     \mathbb P^x(X_t\in B(x,r))\ge 1-c_3.
    \]
    By the Cauchy-Schwarz inequality, one has 
    \[
    (1-c_3)^2\le \mathbb P^x(X_t\in B(x,r))^2=\left(\int_{B(x,r)}p_t(x,z)dm(z)\right)^2\le m(B(x,r)) p_{2t}(x,x).
    \]
    Finally from the choice of $t$ and the uniform volume control assumption, we conclude that there exist constant $c$ (depending on $c_1, c_2$) and $C>0$ (depending on $c_1, c_2$) such that for $0<t<t_0:=c\Psi_2(r_1)$, 
    \[
    p_t(x,x)\ge \frac{C}{\Phi(\Psi_2^{-1}(t))}. 
    \]
\end{proof}

We are now ready to prove \eqref{NDLB}.

\begin{theorem}
Let $t_0$ be as in Lemma \ref{lem:on-diagLB}.  There exists constants $c_1,c_2>0$ such that if  $\Psi_2(c_2d(x,y) )\le t \le t_0$,
\begin{align}
p_t(x,y) \ge \frac{c_1}{\Phi (\Psi_2^{-1}(t))}. 
\end{align}
\end{theorem}
\begin{proof}
Observing that for $t$ small enough (as in Lemma \ref{lem:on-diagLB}) and for $x,y \in X$ such that $d(x,y)<\underline{\iota}$, 
Theorem \ref{thm:upper-Holder} and Lemma \ref{lem:on-diagLB} imply that
\begin{align*}
|p_t(x,y) -p_t(x,x)|
&\le C\left(\frac{d(x,y)}{t}\right)^{1/2} \frac{1}{\Phi (\Psi_2^{-1}(t))^{1/2}}
\le C\left(\Phi (\Psi_2^{-1}(t)) \frac{d(x,y)}{t}\right)^{1/2} p_{t}(x,x).    
\end{align*}   
Taking $\Psi_2(ad(x,y) )\le t$ for $a$ large enough, we obtain again from Lemma \ref{lem:on-diagLB} that 
\begin{equation}\label{eq:near-diagLB}
p_t(x,y)\ge \frac12 p_t(x,x) \ge \frac{c}{\Phi (\Psi_2^{-1}(t))}.    
\end{equation}
\end{proof}

\subsection{Heat kernel characterization of the Sobolev and BV spaces}

We now give yet another characterization of the Sobolev and BV spaces which shows that the theory of heat kernel based Besov classes developed in \cite{BV2,BV3,BV1} can therefore be applied in the context of local trees. The main ingredients of this characterization are the heat kernel estimates \eqref{NDLB} and \eqref{ERE}.  For $f \in L^p(X,m)$ and $t >0$ we denote 
\[
N_{p} (f,t)=\frac{1}{\Psi_p(\Psi_2^{-1}(t))} \int_X \int_X |f(x)-f(y)|^p p_t(x,y) dm(x)dm(y)
\]
and the heat kernel based Besov class is defined by
\[
\mathcal{B}^{p}(X,m)=\left\{ f \in L^p(X,m) \mid \limsup_{t \to 0} N_{p} (f,t) <+\infty  \right\}.
\]

We have then the following theorem.

\begin{theorem}

\

\begin{enumerate}
\item $\mathcal{B}^{1}(X,m)=BV(X,m)$ and there exist constants $c,C>0$ such that for every $f \in BV(X,m)$
\[
c \| \partial f \|(X) \le \sup_{t \in (0,t_0]} N_{1} (f,t) \le C \left( \| f \|_{L^1(X,m)} +\| \partial f \|(X)  \right).
\]
\item For $p >1$, $\mathcal{B}^{p}(X,m)=W^{1,p}(X,m)$ and there exist constants $c,C>0$ such that for every $f \in W^{1,p} (X,m)$
\[
c \int_\mathcal{S} |\partial f |^p d\nu \le \sup_{t \in (0,t_0]} N_{p} (f,t) \le C \left( \| f \|^p_{L^p(X,m)} + \int_\mathcal{S} |\partial f |^p d\nu \right).
\]

\end{enumerate}
\end{theorem}

\begin{proof}
We do the proof for $p=1$, the proof for $p>1$ being similar. We first have for $0<t\le t_0$, from the near diagonal lower bound \eqref{NDLB}
\begin{align*}
N_{1} (f,t) &=\frac{1}{\Psi_1(\Psi_2^{-1}(t))} \int_X \int_X |f(x)-f(y)| p_t(x,y) dm(x)dm(y) \\
 & \ge \frac{1}{\Psi_1(\Psi_2^{-1}(t))} \int_X \int_{B(y,C\Psi_2^{-1}(t))} |f(x)-f(y)| p_t(x,y) dm(x)dm(y) \\
 &\ge \frac{c}{\Psi_1(\Psi_2^{-1}(t)) \Phi (\Psi_2^{-1}(t))} \int_X \int_{B(y,C\Psi_2^{-1}(t))} |f(x)-f(y)| dm(x)dm(y) \\
 &\ge c E_{1,\Psi_1}(f, C\Psi_2^{-1}(t)).
\end{align*}
Therefore, from Theorem \ref{characterization BV} one has
\[
\sup_{t \in (0,t_0]} N_{1} (f,t) \ge C\| \partial f \|(X) .
\]
We now turn to the proof of the upper bound for $N_{1} (f,t)$ where $0<t \le t_0$. Let $0<r<\min (r_0,\underline{\iota}/2)$. We decompose
\[
\int_X \int_X |f(x)-f(y)| p_t(x,y) dm(x)dm(y)=A(t)+B(t)
\]
where $$A(t)=\int_X \int_{X \setminus B(y,r)} |f(x)-f(y)| p_t(x,y) dm(x)dm(y)$$ and $$B(t)=\int_X \int_{ B(y,r)} |f(x)-f(y)| p_t(x,y) dm(x)dm(y).$$ We first estimate $A(t)$ using \eqref{ERE}:
\begin{align*}
A(t) & \le 2 \int_X\left(  \int_{X \setminus B(y,r)}p_t(x,y) dm(x) \right)  |f(y)|  dm(y) \\
 & \le C_1 \exp \left(-C_2 \frac{r}{\Phi^{-1} \left( t/r \right)} \right) \| f \|_{L^1(X,m)}.
\end{align*}
We then estimate $B(t)$ using \eqref{UB}:
\begin{align*}
B(t) & \le \frac{C_1}{\Phi (\Psi_2^{-1}(t))}  \int_X \int_{ B(y,r)} |f(x)-f(y)|  \exp \left(-C_2 \frac{d(x,y)}{\Phi^{-1} \left( t/d(x,y) \right)} \right) dm(x)dm(y) \\
 &\le \frac{C_1}{\Phi (\Psi_2^{-1}(t))} \sum_{k=1}^{+\infty}  \int_X \int_{ B(y,2^{1-k}r)\setminus B(y,2^{-k} r)} |f(x)-f(y)|  \exp \left(-C_2 \frac{d(x,y)}{\Phi^{-1} \left( t/d(x,y) \right)} \right) dm(x)dm(y) \\
 & \le  \frac{C_1}{\Phi (\Psi_2^{-1}(t))} \sum_{k=1}^{+\infty}  \int_X \int_{ B(y,2^{1-k}r)\setminus B(y,2^{-k} r)} |f(x)-f(y)|  \exp \left(-C_2 \frac{2^{-k} r }{\Phi^{-1} \left( 2^{k} t/r \right)} \right) dm(x)dm(y) \\
 & \le \frac{C_1}{\Phi (\Psi_2^{-1}(t))}    \sum_{k=1}^{+\infty} \exp \left(-C_2 \frac{2^{-k} r }{\Phi^{-1} \left( 2^{k} t/r \right)} \right)    \int_X \int_{ B(y,2^{1-k}r)} |f(x)-f(y)| dm(x)dm(y) \\
 & \le \frac{C_1}{\Phi (\Psi_2^{-1}(t))}    \sum_{k=1}^{+\infty} \exp \left(-C_2 \frac{2^{-k} r }{\Phi^{-1} \left( 2^{k} t/r \right)} \right) \Psi_1(2^{1-k}r) \Phi (2^{1-k}r)    E_{1,\Psi_1}(f,2^{1-k}r) \\
  & \le \frac{C_1}{\Phi (\Psi_2^{-1}(t))}  \left(  \sum_{k=1}^{+\infty} \exp \left(-C_2 \frac{2^{-k} r }{\Phi^{-1} \left( 2^{k} t/r \right)} \right) \Psi_1(2^{1-k}r) \Phi (2^{1-k}r) \right) \sup_{\rho \in (0,r]}    E_{1,\Psi_1}(f,\rho)
 \end{align*}
 Using  \eqref{volume assumption}, denoting $\varphi(x)=\max \{ x^\alpha,  x \}$ one then sees that
 
\begin{align*}
 & \sum_{k=1}^{+\infty} \exp \left(-C_2 \frac{2^{-k} r }{\Phi^{-1} \left( 2^{k} t/r \right)} \right) \Psi_1(2^{1-k}r) \Phi (2^{1-k}r)\\
 =& \sum_{k=1}^{+\infty} \exp \left(-C_2 \frac{2^{-k} r }{\Phi^{-1} \left( 2^{k} t/r \right)} \right) \Psi_1(2^{1-k}r)^2 \\
 \le & C_3 \Psi_1 (\Psi_2^{-1}(t))^2 \sum_{k=1}^{+\infty} \exp \left(-C_2 \frac{2^{-k} r }{\Phi^{-1} \left( 2^{k} t/r \right)} \right) \varphi \left(2^{1-k}\frac{r}{\Psi_2^{-1}(t)}\right)^2 \\
 \le &C_4 \Psi_1 (\Psi_2^{-1}(t))^2,
\end{align*}
 where for the last inequality, we used the fact that $\frac{t}{\Phi\circ\Psi_2^{-1}(t)}=\Psi_2^{-1}(t)$ and the inequality
 \[
 \exp \left(-C_2 \frac{2^{-k} r }{\Phi^{-1} \left( 2^{k} t/r \right)} \right)=\exp \left(-C_2 \frac{t }{\Psi_2\circ \Phi^{-1} \left( 2^{k} t/r \right)} \right)\le \exp \left(-C_3  \tilde{\varphi} \left(2^{1-k}\frac{r}{\Psi_2^{-1}(t)}\right) \right)
 \]
 where $ \tilde{\varphi}= \min \{ x^{2/\alpha},  x^{1+\alpha} \}$. We thus have:
 \[
 \frac{C_1}{\Phi (\Psi_2^{-1}(t))}  \left(  \sum_{k=1}^{+\infty} \exp \left(-C_2 \frac{2^{-k} r }{\Phi^{-1} \left( 2^{k} t/r \right)} \right) \Psi_1(2^{1-k}r) \Phi (2^{1-k}r) \right) \le C_3 \Psi_1 (\Psi_2^{-1}(t)).
 \]
 Therefore, one has 
 \begin{align*}
 & \int_X \int_X |f(x)-f(y)| p_t(x,y) dm(x)dm(y) \\
 \le &  C_1 \exp \left(-C_2 \frac{r}{\Phi^{-1} \left( t/r \right)} \right) \| f \|_{L^1(X,m)} +C_3 \Psi_1 (\Psi_2^{-1}(t)) \sup_{\rho \in (0,r]}    E_{1,\Psi_1}(f,\rho).
\end{align*}
Using Theorem \ref{characterization BV} this yields
\[
\sup_{t \in (0,t_0]} N_{1} (f,t) \le C \left( \| f \|_{L^1(X,m)} +\| \partial f \|(X)  \right).
\]
\end{proof}

\subsection{Local heat kernel gradient estimates}

We say that the locally Lipshitz regular condition for harmonic functions holds if there exist $C>0$, $A>1$ and $0<r_2<\underline{\iota}/(6A)$ such that for any ball $B(x_0,r)$ with $r<r_2$, for any $u\in W^{1,2}(X,m)$ which is harmonic in $B(x_0,Ar)$ (i.e. $\int_{\mathcal{S}} \partial u \, \partial v \, d\nu=0$ for every $v\in W^{1,2}(X,m)$ with compact support in $B(x_0,Ar)$), for $m$-a.e. $x,y\in B(x_0,r)$ with $x\ne y$, we have
\begin{equation*}
\frac{|u(x)-u(y)|}{d(x,y)}\le \frac{C}{r}\dashint_{B(x_0,Ar)}|u|d m.
\end{equation*}

By the local tree property, the above condition holds as follows.

\begin{lemma}\label{lem:harmonic Lipschitz}
The locally Lipshitz regular condition holds.
\end{lemma}

\begin{proof}
Let $r>0$ be small enough, $x_0 \in X$ and let $u\in W^{1,2}(X,m)$ be harmonic in $B(x_0,5r)$. Let $z\in X$ and $\varepsilon>0$ be such that $\{ y \in X \mid d(z,y) \le 2 \varepsilon \} \subset B(x_0,5r)$. Let $\Psi_z^\varepsilon$ and $S_{\varepsilon}$ be given as in the proof of Lemma \ref{bump function} (with $z$ instead of $x$). Write $S_{\varepsilon}=\{x_1,\ldots,x_N\}$, where $N$ is bounded above by a constant independent of $z$, $\varepsilon$. We consider the function $\chi=\min (1,2 \Psi_z^\varepsilon)$. As in the proof of Lemma \ref{bump function}, one can check that $\chi \in AC(X)$, that $\chi=1$ on $B(z,\varepsilon/2)$, that $\chi(y)=0$ if $d(y,z) \ge 2\varepsilon$ and that $\nu$ a.e.
\[
\partial_z \chi =-\frac{2}{\varepsilon} 1_{T_\varepsilon}, \quad \text{where} \quad T_\varepsilon=\left\{ y \in \cup_{n=1}^N [ z, x_n] \mid d(z,y) > \varepsilon/2 \right\}.
\]
Moreover, the set $T_\varepsilon$ can be written as a disjoint union of sets $]c_n,x_n]$, $1 \le n \le N$, with $c_n \in [z,x_n]$. Let now $a,b \in B(z,\varepsilon/2)$ with $a \neq b$ and consider the function
\[
w(y)=\frac{d(a,c(a,b,y))}{d(a,b)}, \quad y \in B(z,3\varepsilon).
\]
We have $0 \le w \le 1$ and $\nu$-a.e. $\partial w=0$ outside of $]a,b[$. Since $[a,b] \subset B(z,\varepsilon/2)$, the function $w$ is constant on each of the sets $]c_n,x_n]$ and we denote by $w_n$ this constant. The function $\chi w$ (extended by $0$ outside of $B(z,3\varepsilon)$) belongs to $W^{1,2}(X,m)$ and has a compact support in $B(x_0,5r)$. Since $u$ is harmonic in $B(x_0,5r)$ and $\chi=1$ on $[a,b]$, we therefore have
\[
0=\int_{\mathcal{S}} \chi \, \partial u \, \partial w \, d\nu+\int_{\mathcal{S}} w \, \partial u \, \partial \chi \, d\nu=\frac{u(b)-u(a)}{d(a,b)}-\frac{2}{\varepsilon} \sum_{n=1}^N w_n (u(x_n)-u(c_n)).
\]
We conclude that for every $a,b \in B(z,\varepsilon/2)$
\begin{equation}\label{eq:harmonic Lipschitz}
|u(a)-u(b)| \le \frac{C}{\varepsilon} d(a,b) \sup_{d(z,y) \le \varepsilon} |u(y)|.
\end{equation}
For $0<s\le 4$ we now denote $M(s)=\sup_{B(x_0,sr)} |u|$, which is finite since $u$ is continuous and $(X,d)$ is proper. Let $3 \le s <\sigma \le 4$ and $z \in B(x_0,sr)$. Applying \eqref{eq:harmonic Lipschitz} with $\varepsilon=(\sigma-s)r/2$ yields that for every $0<\eta<1$ and $y \in B(z,\eta \varepsilon/2)$ one has $|u(z)| \le |u(y)|+C\eta M(\sigma)$. Integrating this inequality with respect to $y$ over the ball $B(z,\eta \varepsilon/2)$ and using the uniform volume control and \eqref{volume assumption}, we obtain
\[
M(s) \le \frac{C}{\eta^\alpha (\sigma-s)^\alpha} \dashint_{B(x_0,5r)}|u|d m+C\eta M(\sigma).
\]
Choosing $\eta$ such that $C \eta 2^\alpha \le 1/2$ and iterating this inequality along the sequence $s_k=4-2^{-k}$, $k \ge 0$, we deduce that
\[
M(3) \le C\dashint_{B(x_0,5r)}|u|d m.
\]
Finally, applying \eqref{eq:harmonic Lipschitz} with $z=x_0$ and $\varepsilon=2r$, we obtain that for every $x,y \in B(x_0,r)$ with $x \neq y$
\[
\frac{|u(x)-u(y)|}{d(x,y)}\le \frac{C}{r} M(3) \le \frac{C}{r}\dashint_{B(x_0,5r)}|u|d m,
\]
that is, the locally Lipshitz regular condition holds with $A=5$.
\end{proof}

Then following the same argument as in the proof of \cite[Theorem 1.1]{DRY23} or \cite[``(1)$\Rightarrow$(4)" of Theorem 2.1]{gao2024holderregularityharmonicfunctions}, we have the local heat kernel gradient estimates as follows. Since the setting of those references is different from ours, we give the proof.

\begin{theorem}\label{thm:local gradient}
There exist constants $C_1,C_2>0$ such that for any $0<t\le1$, for any $x\in X$, for $\nu$-a.e. $y\in X$ with $d(x,y)<\underline{\iota}/6$, we have
$$|\partial_yp_t(x,y)|\le\frac{C_1}{t}\exp \left(-C_2 \frac{d(x,y)}{\Phi^{-1} \left( t/d(x,y) \right)} \right).$$
\end{theorem}

\begin{proof}
Let $L$ be the generator of the heat semigroup $P_t$ of the Dirichlet form $\mathcal{E}_2$. We first note that the proof of Lemma \ref{lem:harmonic Lipschitz} also applies to the functions in the domain of $L$. Indeed, let $z\in X$, let $\varepsilon>0$ be small enough and let $\chi$, $w$, $w_n$, $x_n$, $c_n$ be as in that proof. If $u$ is in the domain of $L$, then
\[
-\int_X (Lu) \chi w \, dm=\int_{\mathcal{S}} \chi \, \partial u \, \partial w \, d\nu+\int_{\mathcal{S}} w \, \partial u \, \partial \chi \, d\nu=\frac{u(b)-u(a)}{d(a,b)}-\frac{2}{\varepsilon} \sum_{n=1}^N w_n (u(x_n)-u(c_n)),
\]
so that for every $a,b \in B(z,\varepsilon/2)$
\begin{equation}\label{eq:Poisson Lipschitz}
|u(a)-u(b)| \le d(a,b) \left( \frac{C}{\varepsilon}  \sup_{d(z,y) \le \varepsilon} |u(y)| +\int_{B(z,2\varepsilon)} |Lu| dm \right).
\end{equation}
We now estimate the time derivative of the heat kernel. Let $0<t\le 1$ and $x,y \in X$ with $d(x,y) \le \underline{\iota}/2$. For $t/2 \le s \le t$, we have $p_s(x,y)=\int_X p_{t/8}(x,w) P_{s-t/4} p_{t/8}(y,\cdot)(w) dm(w)$, so that from spectral theory and \eqref{eq:on-diag}
\[
\left| \frac{\partial^2}{\partial s^2} p_s (x,y) \right| \le \frac{C}{t^2} p_{t/4}(x,x)^{1/2} p_{t/4}(y,y)^{1/2} \le \frac{C}{t^2 \Phi (\Psi_2^{-1}(t))}.
\]
On the other hand, from Theorem \ref{thm:exit and UB} and \eqref{volume assumption}, we have for $t/2 \le s \le t$
\[
p_s(x,y) \le \frac{C}{\Phi (\Psi_2^{-1}(t))} \exp \left(-2c \frac{d(x,y)}{\Phi^{-1} \left( t/d(x,y) \right)} \right).
\]
Applying Taylor's formula to the function $s \to p_s(x,y)$ between $t-h$ and $t$, where
\[
h=\frac{t}{2} \exp \left(-c \frac{d(x,y)}{\Phi^{-1} \left( t/d(x,y) \right)} \right),
\]
we therefore obtain
\begin{equation}\label{eq:time derivative}
\left| \frac{\partial}{\partial t} p_t (x,y) \right| \le \frac{C}{t \Phi (\Psi_2^{-1}(t))} \exp \left(-c \frac{d(x,y)}{\Phi^{-1} \left( t/d(x,y) \right)} \right).
\end{equation}
Let now $0<t\le 1$, $x \in X$ and $\varepsilon=\kappa \Psi_2^{-1}(t)$, where $\kappa \in (0,1)$ is small enough. Let $z \in X$ with $d(x,z)<\underline{\iota}/6$. For $y \in B(z,2\varepsilon)$ we have $d(x,y) \le \underline{\iota}/2$ and, since $d(y,z) \le 2 \Psi_2^{-1}(t)$, it follows from \eqref{volume assumption} that
\begin{equation}\label{eq:comparison exp}
\frac{d(x,z)}{\Phi^{-1} \left( t/d(x,z) \right)} \le C \left( 1+\frac{d(x,y)}{\Phi^{-1} \left( t/d(x,y) \right)} \right).
\end{equation}
The function $u=p_t(x,\cdot)=P_{t/2}p_{t/2}(x,\cdot)$ is in the domain of $L$ and $Lu=\frac{\partial}{\partial t} p_t (x,\cdot)$. From \eqref{eq:Poisson Lipschitz}, Theorem \ref{thm:exit and UB}, \eqref{eq:time derivative}, \eqref{eq:comparison exp} and the uniform volume control, we therefore have for every $a,b \in B(z,\varepsilon/2)$
\begin{align*}
|p_t(x,a)-p_t(x,b)| & \le C d(a,b) \left( \frac{1}{\varepsilon \Phi (\Psi_2^{-1}(t))}+\frac{\Phi(2\varepsilon)}{t \Phi (\Psi_2^{-1}(t))} \right) \exp \left(-c \frac{d(x,z)}{\Phi^{-1} \left( t/d(x,z) \right)} \right) \\
 & \le \frac{C}{t} d(a,b) \exp \left(-c \frac{d(x,z)}{\Phi^{-1} \left( t/d(x,z) \right)} \right),
\end{align*}
where we used the fact that $\Psi_2^{-1}(t) \Phi (\Psi_2^{-1}(t))=t$. Thus, for $\nu$-a.e. $y \in B(z,\varepsilon/2)$,
\[
|\partial_yp_t(x,y)|\le\frac{C}{t}\exp \left(-c \frac{d(x,z)}{\Phi^{-1} \left( t/d(x,z) \right)} \right).
\]
Since $X$ is separable, the conclusion follows from \eqref{eq:comparison exp} with the roles of $y$ and $z$ exchanged.
\end{proof}

\section{Global results}

In this last section we deal with the case of a tree with global volume estimates. Namely,  we consider a  complete, locally compact and unbounded tree $(X,d)$. Observe that from Lemma 5.7 in \cite{Kig95} such tree is automatically separable and by Lemma 5.9 in \cite{Kig95} the complete and bounded subsets are compact, i.e. $(X,d)$ is proper. We consider as before a Radon measure $m$ with full support and an increasing function   $\Phi:[0,+\infty) \to[0,+\infty)$  such that

$$c\frac{R}{r}\le\frac{\Phi(R)}{\Phi(r)} \le C \left(\frac{R}{r}\right)^\alpha \text{ for any }r\le R, $$
where $\alpha \ge 1$, and such that for every $x \in X$ and $r>0$, $$c \Phi(r)\le m(B(x,r))\le C\Phi(r).$$ As before, we can assume that $\Phi$ is continuous and strictly increasing. Of course all of our previous results apply to this case, with $r_0$, $\varepsilon_0$ and $\underline{\iota}$ that can be taken arbitrarily large, but we show here how some of them can be improved.

\subsection{Statements of  results}

We first observe that we  recover the T. Kumagai's heat kernel estimate on trees proved in  \cite[Proposition 5.1]{KumagaiPRIMS}.

\begin{theorem}\label{heat kernel global}
The heat kernel $p_t(x,y)$ on $X$ satisfies the following estimates: For every $t >0$ and  $x,y \in X$
\begin{align*}
\frac{C_1}{\Phi (\Psi_2^{-1}(t))} \exp \left(-C_2 \frac{d(x,y)}{\Phi^{-1} \left( t/d(x,y) \right)} \right)\le p_t(x,y) \le \frac{C_3}{\Phi (\Psi_2^{-1}(t))} \exp \left(-C_4 \frac{d(x,y)}{\Phi^{-1} \left( t/d(x,y) \right)} \right).
\end{align*}
\end{theorem}

\begin{proof}
It follows from the arguments from Section \ref{Heat kernel estimates} that the heat kernel  satisfies for every $t >0$ and  $x,y \in X$ the upper bound
\begin{align*}
p_t(x,y) \le \frac{C_3}{\Phi (\Psi_2^{-1}(t))} \exp \left(-C_4 \frac{d(x,y)}{\Phi^{-1} \left( t/d(x,y) \right)} \right).  
\end{align*}
and the near-diagonal estimates: For $\Psi_2(c_2d(x,y) )\le t$,
\begin{align*}
p_t(x,y) \ge \frac{c_1}{\Phi (\Psi_2^{-1}(t))}. 
\end{align*}
Since the metric space $(X,d)$ is geodesic because it is a tree, we deduce by a classical chaining argument, see for instance the proof of \cite[Theorem 3.1]{CoulhonOff}, that the heat kernel also satisfies for every $t >0$ and  $x,y \in X$ the lower bound
\begin{align*}
p_t(x,y) \ge \frac{C_1}{\Phi (\Psi_2^{-1}(t))} \exp \left(-C_2 \frac{d(x,y)}{\Phi^{-1} \left( t/d(x,y) \right)} \right).  
\end{align*}

\end{proof}

By the two-sided heat kernel estimates in Theorem \ref{heat kernel global}, the proof of Theorem \ref{thm:local gradient} now applies to every $t>0$ and every $x,y \in X$, and we have pointwise gradient bound of the heat kernel as follows (see also \cite[Theorem 2.1]{gao2024holderregularityharmonicfunctions} in the case where $\Phi$ is a power function).

\begin{theorem}
There exist constants $C_1,C_2>0$ such that for any $t>0$, for any $x\in X$, for $\nu$-a.e. $y\in X$, we have
$$|\partial_yp_t(x,y)|\le\frac{C_1}{t}\exp \left(-C_2 \frac{d(x,y)}{\Phi^{-1} \left( t/d(x,y) \right)} \right).$$
\end{theorem}

The results below follow from straightforward modifications of the arguments we used before. The proofs are therefore   left to the reader.

\begin{theorem}\label{weak monotonicity Lp global tree}

\

\begin{enumerate}
\item $KS^{\Psi_1,1}(X,m)=\mathcal{B}^{1}(X,m)=BV(X,m)$ and there exist constants $c_1,c_2,c_3,c_4>0$ such that for every $f \in BV(X,m)$
\[
c_1 \sup_{r >0 }E_{1,\Psi_1} (f,r)  \le \| \partial f \|(X) \le c_2\liminf_{r  \to 0 }E_{1,\Psi_1} (f,r) ,
\]
and
\[
c_3 \sup_{t >0} N_{1} (f,t) \le  \| \partial f \|(X)  \le  c_4 \liminf_{t  \to 0} N_{1} (f,t).
\]

\item For $p >1$, $KS^{\Psi_p,p}(X,m)=\mathcal{B}^{p}(X,m)=W^{1,p}(X,m)$ and there exist constants $c_1,c_2,c_3,c_4>0$ such that for every $f \in W^{1,p} (X,m)$
\[
c_1 \sup_{r >0 }E_{p,\Psi_p} (f,r)  \le \int_\mathcal{S} |\partial f |^p d\nu \le c_2\liminf_{r  \to 0 }E_{p,\Psi_p} (f,r) ,
\]
and
\[
c_3 \sup_{t >0} N_{p} (f,t) \le  \int_\mathcal{S} |\partial f |^p d\nu \le  c_4 \liminf_{t  \to 0} N_{p} (f,t).
\]

\end{enumerate}\end{theorem}

\begin{theorem}
Assume  the volume growth function $\Phi$ is given by
\[
\Phi (r)=r^{d_h}
\]
with $d_h \ge 1$. Let  $\alpha_p=1+\frac{d_h-1}{p} $. Then, for   $0\le \alpha \le\alpha_p$
\begin{align*}
\mathcal{B}^{\alpha,p}(X,m)=
\begin{cases}
(L^1(X,m), BV(X,m))_{\alpha/\alpha_1,\infty}, &  p=1 \\
(L^p(X,m), W^{1,p}(X,m))_{\alpha/\alpha_p,\infty}, &  p>1.
\end{cases}
\end{align*}
\end{theorem}

\begin{theorem}\label{Nash global tree} 
Assume the volume growth function $\Phi$ is given by
\[
\Phi (r)=r^{d_h}
\]
with $d_h \ge 1$. Let $p>1$. The following Nash inequality holds for every $f \in W^{1,p}(X,m)$,
\[
\| f \|_{L^p(X,m)} \le C \| \partial f \|_{L^p(\mathcal S,\nu)}^{\theta} \| f \|^{1-\theta}_{L^1(X,m)}
\]
where $\theta=\frac{(p-1)d_h}{p-1+pd_h}$. 
\end{theorem}

\subsection{$L^p$ heat kernel gradient bounds}

The goal of this section is to establish $L^p$ gradient bounds for the heat semigroup $P_t$ of the Dirichlet form $\mathcal{E}_2$. We recall that $L$ denotes the generator of $P_t$ and, for a bounded function $f$, we set $LP_t f(x)=\int_X \frac{\partial}{\partial t} p_t (x,y) f(y) dm(y)$, which is consistent with the action of $LP_t$ on $L^2(X,m)$. We start with the following lemma.

\begin{lemma}\label{Estimate LP}
\

\begin{enumerate}
\item Let $p\ge 1$. There exists a constant $C>0$ such that for every bounded function $f \in AC^{p}(X)$ (in particular for every $f \in W^{1,p}(X,m)$) and $t>0$
\[
\| L P_t f \|_{L^p(X,m)} \le  C\frac{ \Psi_2^{-1}( t)^{1-\frac{2}{p}} }{ t^{1-\frac{1}{p}}}  \| \partial  f \|_{L^p(\mathcal{S},\nu)}
\]
\item There exists a constant $C>0$ such that for every $f \in W^{1,\infty}(X,m)$ and $t>0$
\[
\| L P_t f \|_{L^\infty(X,m)} \le  C\frac{ \Psi_2^{-1}( t)}{ t}  \| \partial  f \|_{L^\infty(\mathcal{S},\nu)}
\]

\end{enumerate}
\end{lemma}

\begin{proof}
We first assume $p <\infty$. The first ingredient is the estimate 
\begin{align*}
\| P_t f -f \|^p_{L^p(X,m)} \le \int_X \int_X p_t(x,y)|f(x)-f(y)|^p dm(x)dm(y)  \le C \Psi_p(\Psi_2^{-1}(t) ) \| \partial  f \|^p_{L^p(\mathcal{S},\nu)}
\end{align*}
which follows from the fact that the semigroup $P_t$ is conservative (this is seen by letting $r \to +\infty$ in \eqref{eq:exit-prob}, which now holds for every $r>0$) and from Theorem \ref{weak monotonicity Lp global tree}, the proof of which does not use the fact that $f \in L^p(X,m)$. The second ingredient is the estimate
\[
\| L P_t h \|_{L^p(X,m)} \le \frac{C}{t} \|  h \|_{L^p(X,m)}, \quad h \in L^p(X,m),
\]
which for $p>1$ follows from analyticity of the semigroup and which for $p=1$ follows from the following upper bound
\begin{align}\label{Davies estimate}
\left| \frac{\partial}{\partial t} p_t (x,y) \right| \le \frac{C}{t} p_{ct}(x,y)
\end{align}
which is a consequence of \cite{MR1423289} and Theorem \ref{heat kernel global}. Since $f$ is bounded, it also follows from \eqref{Davies estimate} that $\| LP_t f \|_{L^\infty(X,m)} \le \frac{C}{t} \| f \|_{L^\infty(X,m)}$, so that $\lim_{t \to +\infty} \| LP_t f \|_{L^\infty(X,m)} =0$. Then, we have by the semigroup property of $P_t$ that
\begin{align*}
\| LP_{t} f \|_{L^p(X,m)}  =\left\| \sum_{k=1}^{\infty}( LP_{2^k t} f -  LP_{2^{k-1} t} f )  \right \|_{L^p(X,m)} 
 & \le  \sum_{k=1}^{\infty} \left\| LP_{2^k t} f -  LP_{2^{k-1} t} f  \right \|_{L^p(X,m)} \\
 & \le \sum_{k=1}^{\infty} \left\| LP_{2^{k-1} t} (P_{2^{k-1}t} f -  f)  \right \|_{L^p(X,m)} \\
 &\le C \sum_{k=1}^{\infty} \frac{1}{2^{k-1} t}  \left\| P_{2^{k-1}t} f -  f  \right \|_{L^p(X,m)} \\
 &\le C  \sum_{k=1}^{\infty} \frac{ \Psi_p(\Psi_2^{-1}(2^{k-1} t))^{1/p}}{2^{k-1} t}  \| \partial  f \|_{L^p(\mathcal{S},\nu)} \\
 &\le C\frac{ \Psi_p(\Psi_2^{-1}( t))^{1/p}}{ t}   \| \partial  f \|_{L^p(\mathcal{S},\nu)}.
\end{align*}
For $p=\infty$, one uses instead the estimate
\begin{align*}
|P_tf(x)-f(x)| & \le \int_X p_t(x,y) | f(x)-f(y)| dm(y) \\
 & \le  \int_X p_t(x,y) d(x,y) dm(y) \| \partial f \|_{L^\infty(\mathcal{S},\nu)} \\
 & \le C \Psi_2^{-1}(t)  \| \partial f \|_{L^\infty(\mathcal{S},\nu)}
\end{align*}
and the estimate
\[
\| L P_t f \|_{L^\infty(X,m)} \le \frac{C}{t} \|  f \|_{L^\infty(X,m)}
\]
which also follows from \eqref{Davies estimate}. The rest of the proof proceeds then as before.
\end{proof}

We can finally  state the $L^p$-gradient bound estimates.

\begin{theorem}
\

\begin{enumerate}
\item Let $p \ge 1$. There exists a constant $C>0$ such that for every $f \in L^{p}(X,m)$ and $t>0$
\[
\| \partial P_t f \|_{L^p(\mathcal{S},\nu)} \le  C\frac{ \Psi_2^{-1}( t)^{-1+\frac{2}{p}}}{ t^{1/p}}  \|  f \|_{L^p(X,m)}
\]
\item There exists a constant $C>0$ such that for every $f \in L^{\infty}(X,m)$ and $t>0$
\[
\| \partial P_t f \|_{L^\infty(\mathcal{S},\nu)} \le  \frac{ C}{ \Psi_2^{-1}( t) }  \|  f \|_{L^\infty(X,m)}.
\]
\end{enumerate}

\end{theorem}

\begin{proof}
Let $f$ be a piecewise affine function on $X$ meaning that $f \in AC(X)$ with $\nu$-a.e.
\begin{align}\label{affine}
\partial f =\sum_{i=1}^n \alpha_i 1_{]a_i,b_i[},
\end{align}
where $\alpha_i \in \mathbb{R}$ and $a_i,b_i \in X$. Note that such a function $f$ is bounded and that any function of the type $\sum_{i=1}^n \alpha_i 1_{]a_i,b_i[}$ is the weak gradient of a piecewise affine function. We have then for every $g \in L^1(X,m) \cap L^\infty(X,m)$
\begin{align*}
 \int_X g LP_t f dm  &= -\int_{\mathcal{S}} \partial P_t g \, \partial f d\nu.
\end{align*}
Indeed, if $f \in W^{1,2}(X,m)$ this follows from $\int_X g LP_t f dm=\int_X (LP_t g) f dm=-\int_{\mathcal{S}} \partial P_t g \, \partial f d\nu$ and in general it follows by applying this identity to the functions $f \Psi_{x_0}^R$, where $x_0 \in X$ and $\Psi_{x_0}^R$ is as in Lemma \ref{bump function}, and letting $R \to +\infty$. Let now $1 \le p \le \infty$ and let $q$ be the conjugate exponent. We have from Lemma \ref{Estimate LP}
\[
\left| \int_{\mathcal{S}} \partial P_t g \, \partial f d\nu \right| \le \| g \|_{L^q(X,m)}  \| LP_t f \|_{L^p(X,m)} \le C\frac{ \Psi_2^{-1}( t)^{1-\frac{2}{p}} }{ t^{1-\frac{1}{p}}}    \| g \|_{L^q(X,m)} \| \partial  f \|_{L^p(\mathcal{S},\nu)}.
\]
Since it is true for every functions $f$ of the type \eqref{affine}, and since the functions $\sum_{i=1}^n \alpha_i 1_{]a_i,b_i[}$ are dense in $L^p(\mathcal{S},\nu)$ if $p<+\infty$ and can be used to compute by duality the norm of $L^1(\mathcal{S},\nu)$ if $p=+\infty$, this allows to conclude
\[
\| \partial P_t g \|_{L^q(\mathcal{S},\nu)} \le  C\frac{ \Psi_2^{-1}( t)^{1-\frac{2}{p}} }{ t^{1-\frac{1}{p}}}  \|  g \|_{L^q(X,m)}
\]
This holds for every $g \in L^1(X,m) \cap L^\infty(X,m)$ and therefore, by density, for every $g \in L^q(X,m)$ if $q<+\infty$. If $q=+\infty$ and $g \in L^\infty(X,m)$, we apply the previous estimate to the functions $g_n=g 1_{B(x_0,n)} \in L^1(X,m) \cap L^\infty(X,m)$, where $x_0 \in X$. From Theorem \ref{thm:Morrey}, the functions $P_t g_n$ are Lipschitz continuous with a Lipschitz constant bounded by $\frac{C}{\Psi_2^{-1}(t)} \| g \|_{L^\infty(X,m)}$. Since $P_t g_n$ pointwisely converges to $P_t g$ when $n \to +\infty$, the function $P_t g$ satisfies the same Lipschitz estimate, which implies the gradient bound.
\end{proof}

\vspace{5pt}
\noindent
\begin{minipage}{\textwidth}
    \small
    \textbf{Fabrice Baudoin:} \\
    Department of Mathematics, Aarhus University \\
    Email: fbaudoin@math.au.dk
\end{minipage}

\vspace{10pt} 

\noindent
\begin{minipage}{\textwidth}
    \small
    \textbf{Li Chen:} \\
    Department of Mathematics, Aarhus University \\
    Email: lchen@math.au.dk
\end{minipage}

\vspace{10pt} 

\noindent
\begin{minipage}{\textwidth}
    \small
    \textbf{Meng Yang:} \\
    Department of Mathematics, Aarhus University \\
    Email: yang@math.au.dk
\end{minipage}

\end{document}